\documentclass[11pt]{amsart}

\usepackage{amsmath,amsthm,amsfonts,amssymb,bm,graphicx}
\usepackage{mathrsfs}
\usepackage[alphabetic,initials]{amsrefs}
\usepackage[colorlinks=true,citecolor=blue,linkcolor=blue,urlcolor=blue,
pdfstartview=FitH]{hyperref}

\usepackage[margin=1.1in]{geometry}

\theoremstyle{plain}
\newtheorem{theorem}{Theorem}[section]

\newtheorem{lemma}{Lemma}[section]

\theoremstyle{definition}
\newtheorem*{acknowledgments}{Acknowledgement}   
\newtheorem{remark}{Remark}[section]
\newtheorem*{notation}{Notation and conventions}

\renewcommand{\geq}{\geqslant}
\renewcommand{\leq}{\leqslant}

\DeclareMathOperator{\GL}{GL}

\newcommand{\sym}{\mathrm{sym}}

\newcommand{\Kl}{\operatorname{Kl}}

\numberwithin{equation}{section}

\makeatletter
\def\Ddots{\mathinner{\mkern1mu\raise\p@
\vbox{\kern7\p@\hbox{.}}\mkern2mu
\raise4\p@\hbox{.}\mkern2mu\raise7\p@\hbox{.}\mkern1mu}}
\makeatother
\makeatletter
\DeclareRobustCommand\widecheck[1]{{\mathpalette\@widecheck{#1}}}
\def\@widecheck#1#2{%
    \setbox\z@\hbox{\m@th$#1#2$}%
    \setbox\tw@\hbox{\m@th$#1%
       \widehat{%
          \vrule\@width\z@\@height\ht\z@
          \vrule\@height\z@\@width\wd\z@}$}%
    \dp\tw@-\ht\z@
    \@tempdima\ht\z@ \advance\@tempdima2\ht\tw@ \divide\@tempdima\thr@@
    \setbox\tw@\hbox{%
       \raise\@tempdima\hbox{\scalebox{1}[-1]{\lower\@tempdima\box
\tw@}}}%
    {\ooalign{\box\tw@ \cr \box\z@}}}
\makeatother

\begin{document}

\title[Rankin--Selberg coefficients in arithmetic progressions]
{Rankin--Selberg coefficients in arithmetic progressions modulo prime powers}

\begin{abstract}
Let $\varepsilon>0$ be given. For prime power moduli $q=p^k$ with $k\geq 2$ and $p\neq 3$, and assuming the Ramanujan--Petersson conjecture for $\GL_2$ Maass forms, we prove that the Rankin--Selberg coefficients $\{\lambda_f(n)^2\}_{n\geq 1}$ have a level of distribution 
$\theta=2/5+3/305-\varepsilon$ in arithmetic progressions $n \equiv a \bmod q$.
\end{abstract}

\author{Tengyou Zhu}
\address{School of Mathematics, Shandong University
                       \\Jinan, Shandong 250100, China}
\email{zhuty@mail.sdu.edu.cn}

\date{\today}

\keywords{Rankin--Selberg coefficients, Arithmetic progressions,
           hyper-Kloosterman sums, delta method}
\subjclass[2020]{11B25, 11F11, 11F30, 11L05, 11N37, 11T23}
\maketitle	

\section{Introduction}\label{Intr.}

The study of arithmetic functions along progressions is one of the most important 
and central problems in number theory.
Given an arithmetic function $g:\mathbb{N}\rightarrow\mathbb{C}$, 
one can study the distribution of $g(n)$ by asking how uniformly 
the values $g(n)$ are distributed when we vary $n$ over arithmetic progressions $n \equiv a \bmod q$.
The goal is to show that, as $X\rightarrow \infty$, the asymptotic formula
\begin{equation}\label{level-of-distribution}
\mathop{\sum_{n\leq X}}_{n\equiv a \bmod q}g(n)-\frac{1}{\varphi(q)}\mathop{\sum_{n\leq X}}_{(n,q)=1}g(n)\ll_A \frac{X}{q}\left(\log X\right)^{-A}
\end{equation}
holds for $q\leq X^{\theta-\varepsilon}$,
with $\theta$ as large as possible.
Here $\varphi$ is the Euler totient function and the exponent
$\theta$ is called the level of distribution.
Similar to the Elliott--Halberstam conjecture \cite{EH}
for the distribution of primes in arithmetic progressions,
it is widely believed that one can take $\theta=1$.

As $q$ becomes large, taking $g(n)=\Lambda(n)$, the von Mangoldt function, 
the classical Siegel--Walfisz theorem implies that \eqref{level-of-distribution}
holds for $q\leq (\log X)^{A}$ for any constant $A$, whereas the Generalized Riemann Hypothesis
predicts $q\leq X^{1/2-\varepsilon}$. The significant result is that the Bombieri--Vinogradov theorem
confirms this prediction on average over the moduli.

The problem is interesting and meaningful
when $g(n)$ is the sequence of Fourier coefficients of some automorphic form.
Let
$$
L(\pi, s)=\sum_{n\geq 1}\frac{\lambda_{\pi}(n)}{n^s}=\prod_p L(\pi_p,s), \quad \Re(s)>1
$$
be an automorphic $L$-function of some degree $d\geq 2$.

For $\pi$ a $\GL_d$ automorphic form, by applying the functional equation
for the twisted $L$-function $L(\pi\otimes\chi, s)$ and appealing to
Deligne's bound
\begin{equation}\label{Deligne estimate}
|\Kl_d(a;q)|\leq d^{\omega(q)},
\end{equation}
possibly subject to the Ramanujan conjecture $|\lambda_\pi(n)|\leq n^{\varepsilon}$,
one can obtain a level of distribution $\theta=\frac{2}{d+1}$ (cf.~\cite[Remark 2]{HLW}).
Here $\chi \bmod q$ is a primitive Dirichlet character and
$$
\Kl_d(a;q)=\frac{1}{q^{(d-1)/2}}
\sum_{\substack{x_1,\ldots,
x_d\in \mathbb{Z}/q \mathbb{Z}\\x_1\cdots
x_d=a}}e\Bigl(\frac{x_1+\cdots+x_d}{q}\Bigr),
$$
is the $d$-th hyper-Kloosterman sum.
According to~\cite[Section 1.1]{KLM}, such an exponent is regarded as \emph{the standard level of distribution} of the sequence $\{\lambda_{\pi}(n)\}_{n\geq1}$ for individual moduli. To beat this exponent for various $g$ is an active area of research. See~\cite{FI, KMS} for some cases where a better level of distribution is known.

More recently,
Kowalski--Lin--Michel~\cite[Theorem 1.3]{KLM} proved the exponent of distribution $\theta=2/5+1/260-\varepsilon$
for the Rankin--Selberg coefficients $\{\lambda_f(n)^2\}_{n\geq1}$ to prime moduli, where $f$ is a $\GL_{2,\mathbb{Q}}$ automorphic representation. Their key input is to obtain non-trivial bounds for the correlation sums (see \cite[Theorem 1.4]{KLMS})
$$
S_{V}(K;X)=\sum_{n\geq1}\lambda_\pi(n) K(n)V\left(\frac{n}{X}\right),
$$
where $\pi$ is a $\GL_{3,\mathbb{Q}}$ automorphic representation and $K(n)$ is the trace function associated to a suitable
$\ell$-adic middle extension sheaf $\mathcal{F}$ on the
affine line $\mathbf{A}_{\mathbb{F}_q}^1$.

In this paper, inspired by the work of Kowalski--Lin--Michel \cite{KLM}, we obtain the following result for the Rankin--Selberg coefficients $\{\lambda_f(n)^2\}_{n\geq1}$ for the prime-power moduli case.

\begin{theorem}\label{main-theorem1}
Let $f$ be either a holomorphic Hecke eigencuspform or a Hecke--Maass cusp form of full level whose Hecke eigenvalues are denoted $\{\lambda_f(n)\}_{n\geq 1}$ and let $\varpi_f\geq 0$ be such that for any $\varepsilon>0$ the following bound  holds
\begin{equation}\label{RPbound}
    |\lambda_f(n)|\leq n^{\varpi_f+\varepsilon} \quad  (n \rightarrow \infty) 
\end{equation}
and let
\begin{equation*}
   \theta_{f}=\frac{25}{61(1+4\varpi_f)}.
\end{equation*}
Let  $q=p^k$ for a prime $p\neq 3$, $k\geq2$,
and $a\geq 1$ be an integer such that $(a,q)=1$.
For any $X\geq 1$  and $\varepsilon>0$ satisfying
\begin{equation}\label{qbound}
q\leq X^{\theta_f-\varepsilon},
\end{equation}
we have
$$
\sum_{\substack{n\geq 1 \\ n\equiv a\bmod q}}\lambda_f(n)^2V\left(\frac{n}{X}\right)-\frac{1}{\varphi(q)}\sum_{\substack{n\geq 1 \\ (n,q)=1}}\lambda_f(n)^2V\left(\frac{n}{X}\right)\ll_{f,V,\varepsilon} p^{3/4}(X/q)^{1-\delta}
$$
for some $\delta=\delta(\varepsilon)>0$. In particular if $f$ is holomorphic then $\varpi_f=0$ and
$$
\theta_f=25/61=2/5+3/305.
$$
\end{theorem}

\begin{remark}
It is unfortunate that, just like in~\cite{KLM}, plugging the Kim--Sarnak bound
$\varpi_f \leq 7/64$ (\cite[Appendix 2]{Kim})
is not sufficient to ensure $\theta_f > 2/5$.
To guarantee $\theta_f > 2/5$, we need to assume $\varpi_f \leq 3/488$.
\end{remark}

\begin{remark}
There is an analytic analog of these questions. In the works \cite{huang, huang1},
Huang proved the following bound for the sharp-cut sum
$$
\sum_{n\leq X}\lambda_f(n)^2
- \frac{L(\sym^2 f, 1)}{\zeta(2)}X = O\bigl( X^{3/5-3/305+\varepsilon}\bigr),
$$
where $L(\sym^2 f, 1)$ denotes the symmetric square $L$-function of $f$, resolving a long standing problem going back to Rankin and Selberg.
\end{remark}

To prove Theorem \ref{main-theorem1}, we first consider
$$
L(f\times f, s):=\zeta(2s)\sum_{n\geq 1}\frac{\lambda_f(n)^2}{n^s}, \quad \Re s > 1,
$$
the Rankin--Selberg $L$-function of $f \times f$. Note that $f \times f = 1\boxplus{\sym^2 f}$,
where ${\sym^2 f}$ denotes the $\GL_{3,\mathbb{Q}}$ symmetric square lift of (the automorphic representation attached to) $f$.
Therefore we obtain
\begin{equation}\label{convident}
L(f\times f, s)=\zeta(s)\sum_{m\geq 1}\frac{\lambda_{{\sym^2 f}}(m)}{m^s}=:\zeta(s)L({\sym^2 f}, s).
\end{equation}
This leads us to consider the more general case $\phi$, where $\phi$ is a Hecke--Maass cusp form (not necessarily symmetric) for $\GL_3$ and $\{A(1, n)\}_{n\geq 1}$ are the coefficients of its standard $L$-function $L(\phi, s)$. Set
$$
\lambda_{1\boxplus\phi}(n):=1\star A(1, n):=\sum_{\ell m = n}A(1, m).
$$
Theorem \ref{main-theorem1} is then a simple consequence of

\begin{theorem}\label{main-theorem2}
Let  $q=p^k$ for a prime $p\neq 3$, $k\geq2$, and $a\geq 1$ be an integer such that $(a,q)=1$. For any $\varepsilon>0$ and $X$ satisfying
\begin{equation}
    \label{qbound2}
    q\leq X^{25/61-\varepsilon}=X^{2/5+3/305-\varepsilon},
\end{equation}
we have
$$
\sum_{\substack{n\geq 1 \\ n\equiv a\bmod q}}\lambda_{1\boxplus\phi}(n)V\left(\frac{n}{X}\right)-\frac{1}{\varphi(q)}\sum_{\substack{n\geq 1 \\ (n,q)=1}}\lambda_{1\boxplus\phi}(n)V\left(\frac{n}{X}\right)\ll_{\phi,V,\varepsilon} p^{3/4}(X/q)^{1-\delta}
$$
for some $\delta=\delta(\varepsilon)>0$.	
\end{theorem}

\begin{remark}
Unlike \eqref{qbound}, the exponent in \eqref{qbound2}
is independent of any approximation to the Ramanujan--Petersson
conjecture for $\GL_{3}$ automorphic representations.
\end{remark}

In order to prove Theorem \ref{main-theorem2}, we will use a power
saving bound for the analytic twisted sum of $\GL_3$ Fourier coefficients.
Define
$$
S(N):=\sum_{m \geq 1}A(1, m)\Kl_4(m\ell; q)V\left(\frac{m}{N}\right),
$$
where $\ell \in \mathbb{Z}$, and $q=p^k$ with $p \ne 3$ and
$k\geq 2$.

\begin{theorem}\label{Bound theorem}
With notation as above, we have
\begin{equation*}
 S(N)\ll p^{3/4}N^{3/4+\varepsilon}q^{3/10}+N^{1/2+\varepsilon}q^{13/20}.
\end{equation*}
\end{theorem}

\begin{remark}
Sun and Zhao \cite{SZ} obtained subconvexity bounds for $L(\phi \times \chi, 1/2)$, where $\chi$ is a primitive Dirichlet character of prime power
conductor $q=p^k$. By the approximate functional equation, they need to consider
$$
\sum_{m \geq 1}A(1, m)\chi(m)V\left(\frac{m}{N}\right).
$$
This kind of situation is simpler, because one can take advantage
of the fact that the Dirichlet character is multiplicative.
\end{remark}

\begin{notation}
Throughout the paper, the letters
$\varepsilon$ and $A$ denote arbitrarily small and arbitrarily large positive
real numbers, respectively, not necessarily the same at each occurrence.
As usual, $e(z) = e^{2 \pi i z}$.
\end{notation}

\begin{acknowledgments}
The author thanks Professor Jianya Liu for his help and encouragement, 
and gratefully acknowledges the many helpful suggestions from Professor 
Yongxiao Lin during the preparation of this paper.
\end{acknowledgments}

\section{Preliminaries}
\label{prelim}
In this section, we present some essential information and tools needed later. Firstly, we briefly
review some basic facts of automorphic $L$-functions on $\GL_3$.

\subsection{Automorphic $L$-functions}\label{2.1}

Let $\phi$ be a Hecke--Maass cusp form of
type $\nu=(\nu_1,\nu_2)$ for $\operatorname{SL}_3(\mathbb{Z})$, which is an
eigenfunction for all the Hecke operators.
Let the Fourier coefficients be $A(n_1,n_2)$,
normalized so that $A(1, 1)=1$.

By Rankin--Selberg theory, we have
\begin{align}\label{GL3-Rankin--Selberg}
\mathop{\sum\sum}_{n_1^2n_2\leq N} \left|A(n_1,n_2)\right|^2 \ll N.
\end{align}
We denote  $\kappa=0$ if $\chi(-1)=1$ and $\kappa=1$ if $\chi(-1)=-1$ and denote
 $$\tau_\chi=q^{-1/2}\sum_{x\in \mathbb{F}_q}\chi(x)e\Big(\frac{x}q\Big)$$
the normalized Gauss sum.
The $L$-function
$$L((1\boxplus\phi)\times \chi,s)=\sum_{n\geq 1}
\frac{\lambda_{1\boxplus\phi}(n)\chi(n)}{n^s}
=L(\chi,s)L(\pi\times \chi,s)
$$
has analytic continuation to $\mathbb{C}$
and satisfies a functional equation of the form:
\begin{equation*}
\Lambda((1\boxplus\phi)\times \chi,s)=\tau_\chi^4\Lambda((1\boxplus\phi)
\times \overline{\chi},1-s)
\end{equation*}
where
$$\Lambda((1\boxplus\phi)\times \chi,s)=
q^{2s}L_\infty(1\boxplus\phi,s+\kappa)L((1\boxplus\phi)\times \chi,s)
$$
is the completed $L$-function and
 $$
 L_\infty(1\boxplus\phi,s)=\prod_{i=1}^4\Gamma_\mathbb{R}
 (s-\iota_{i}),\ \Gamma_\mathbb{R}(s)=\pi^{-s/2}\Gamma(s/2)$$
with
$$\{\iota_{i},\ i=1,2,3,4\}=\{0,\nu_2-\nu_1,2\nu_1+\nu_2-1,1-\nu_1-2\nu_2\}$$
denoting the local Archimedean factor of $1\boxplus\phi$.
For more details, we refer the readers to Iwaniec--Kowalski \cite[Section 5.1]{IK}.

\subsection{Summation formulas}

We first recall the Poisson summation formulae over
an arithmetic progression.
\begin{lemma}\label{Fourier transform}
Let $\beta \in \mathbb{Z}$ and $c \in \mathbb{Z}_{\geq 1}$. For a Schwartz function $w: \mathbb{R} \to \mathbb{C}$, we have
\begin{align*}
\mathop{\sum}_{\substack{ n\in \mathbb{Z} \\ n \equiv \beta \bmod c}} w(n)=\frac{1}{c}
\sum_{n\in\mathbb{Z}}\widehat{w}\left( \frac{n}{c}\right)e\left(\frac{n\beta}{c}\right),
\end{align*}
where $$\widehat{w}(y)=\int_{\mathbb{R}}w(x)e(-xy)\mathrm{d}x$$ is the Fourier transform of $w(x)$.
\begin{proof}
See e.g.\cite[Eq. (4.24)]{IK}.
\end{proof}
\end{lemma}
In our work we crucially require the $\GL_3$ Voronoi summation formula proved first by Miller and Schmid \cite{MS2}, which we now state.

Denote the Langlands parameters of $\phi$ by
$$\mu_1=-\nu_1-2\nu_2+1, \quad \mu_2=-\nu_1+\nu_2,\quad \mu_3=2\nu_1+\nu_2-1.$$
We define
\begin{align}\label{Gamma1}
\gamma_{\pm} (s)=\frac{1}{2\pi^{3(s+1/2)}}\left\{\prod_{j=1}^3
\frac{\Gamma\left(\frac{1+s+\mu_j}{2}\right)}
{\Gamma\left(\frac{-s-\mu_j}{2}\right)}\mp
i \prod_{j=1}^3
\frac{\Gamma\left(\frac{2+s+\mu_j}{2}\right)}
{\Gamma\left(\frac{1-s-\mu_j}{2}\right)}\right\}.
\end{align}
For $w(x)\in C_c^\infty(0,\infty)$ we
denote by $\widetilde{w}(s)$
the Mellin transform of $w(x)$.
Let
\begin{align}\label{intgeral transform-3}
\mathcal{W}^{\pm}\left(x\right)
=\frac{1}{2\pi i}\int_{(\sigma)}x^{-s}
\gamma_{\pm}(s)\widetilde{w}(-s)\mathrm{d}s,
\end{align}
where $\sigma>\max\limits_{1\leq j\leq 3}\{-1-\mathrm{Re}(\mu_j)\}$.
Then we have the following Voronoi summation formula.

\begin{lemma}\label{voronoiGL3}
Let $c\in \mathbb{N}$ and $u\in \mathbb{Z}$ be such
that $(u,c)=1$. Then
\begin{equation*}
\sum_{n=1}^{\infty}A\left(1,n\right)
e\left(\frac{un}{c}\right)w\left(n\right)
=c\sum_{\pm}\sum_{n_{1}|c}
\sum_{n_{2}=1}^{\infty}\frac{A\left(n_{2},n_{1}\right)}{n_{1}n_{2}}
S\left(\overline{u},\pm n_{2};\frac{c}{n_{1}}\right)
\mathcal{W}^{\pm}\left(\frac{n_{1}^{2}n_{2}}{c^{3}}\right),
\end{equation*}
where $u \overline{u} \equiv 1(\bmod\,c)$ and $S(m,n;c)$ is the classical Kloosterman sum.
\end{lemma}

\subsection{The delta method}

Our method is based on separating two oscillatory coefficients using the $\delta$-method. In the present situation we will use a version of the
circle method by Duke, Friedlander and Iwaniec (see \cite[Chapter 20]{IK}).

For any $n\in \mathbb{Z}$ and $C\in \mathbb{R}^+$, we have
\begin{align}\label{DFI's}
\delta_{n=0}=\frac{1}{C}\sum_{1\leq c\leq C} \;\frac{1}{c}\;\sideset{}{^*}\sum_{u\bmod{c}}
e\left(\frac{nu}{c}\right)\int_\mathbb{R}g(c,\zeta) e\left(\frac{n\zeta}{cC}\right)\mathrm{d}\zeta,
\end{align}
where the $*$ on the sum indicates
that the sum over $u$ is restricted to $(u,c)=1$.
The function $g$ has the following properties
(see (20.158) and (20.159) of~\cite{IK} and~\cite[Lemma 15]{huang})
\begin{align}\label{g-h}
g(c,\zeta)\ll |\zeta|^{-A},\quad g(c,\zeta) = 1 +
O\left(\frac{C}{c}\left(\frac{c}{C}+|\zeta|\right)^A\right)
\end{align}
for any $A>1$ and
\begin{align}\label{g rapid decay}
\frac{\partial^j}{\partial \zeta^j}g(c,\zeta)\ll
|\zeta|^{-j}\min\left(|\zeta|^{-1},\frac{C}{c}\right)\log C, \quad j\geq 1.
\end{align}
In particular the first property in \eqref{g-h} implies that
the effective range of the integration in
\eqref{DFI's} is $[-C^\varepsilon, C^\varepsilon]$.

\subsection{Evaluations of Kloosterman sums}

We now turn to Kloosterman sums, and the following lemma characterizes when such sums can vanish.
\begin{lemma}\label{lem:Kloosterman=0-initial}
For a prime $p$ and an integer $k \geq 2$, let $q=p^k$.
For any $m\in \mathbb{Z}$ with $p\mid m$ and $(a,p)=1$, we have
$$
\Kl_4(am;q)=0.
$$
\end{lemma}

\begin{proof}
Noting
$$
\Kl_4(am;p^k)=\frac{1}{p^{k/2}}\,\,\sideset{}{^*}
\sum_{x\bmod p^k}e\left(\frac{am\overline{x}}{p^k}\right)\Kl_3(x; p^k),
$$
we have
$$
\Kl_4(am;p^k)=\frac{1}{p^{k/2}}\,\,\sideset{}{^*}\sum\limits_{{y \bmod p^{k-1}}} \sum_{x \bmod p} e\left(\frac{am\,\overline{y + p^{k-1} x}}{p^k}\right)\Kl_3(y + p^{k-1} x; p^k).
$$
Summing over $x$, we get
$$
\Kl_4(am;p^k)=\frac{1}{p^{k/2}}\,\,\sideset{}{^*}\sum\limits_{{y \bmod p^{k-1}}} e\left(\frac{am y^{-1}+y}{p^k}\right)\Kl_3(y ; p^k) \sum_{x \bmod p} e\left(\frac{x}{p}\right)
$$
since $p\mid m$. The inner sum over $x$ now vanishes due to orthogonality of additive characters. This completes the proof.
\end{proof}
We use the following evaluation of the normalised hyper-Kloosterman sums modulo prime powers,
which can be found in~\cite[Eq. (1.16)]{DB}.
\begin{lemma}\label{lem:Kloosterman1}
Let $q=p^k$ with a prime $p\neq 3$ and $k\geq 2$. Then for $a\in \mathbb{Z}_p^{\times}$, we have
\begin{equation*}
\Kl_3(a;q)=\left(\frac{p^k}{3}\right)
\sum_{\substack{r \in \mathbb{Z}_p \\ r^3= a}}e\left(\frac{3r}{p^k}\right),
\end{equation*}
where $\left(\frac{p^k}{3}\right)$ is the Jacobi symbol.
\end{lemma}
\begin{remark}\label{solution}
We deduce from Lemma \ref{lem:Kloosterman1} that $\Kl_3(a;q)=0$ unless $r\in \mathbb{Z}_p^{\times3}$ (i.e., the cube of a residue class prime $p$).
Now, depending on whether $p\equiv1 \bmod 3$ or $p\equiv2 \bmod 3$, if $r^3=a \in \mathbb{Z}_p^{\times}$, then there are either
\begin{itemize}
\item[-]
exactly 3 critical points, when $p\equiv1 \bmod 3$, or
\vskip 2mm
\item[-]
 exactly 1 critical point, when $p\equiv2 \bmod 3$.
\end{itemize}
\end{remark}

The proof of Theorem \ref{main-theorem2} consists of getting square-root cancellations in certain character sums which we record here for convenience.

Let $p \ne 3$ be a prime and $m\in\mathbb{Z}$. Suppose $\gamma_1,\gamma_2$ are integers such that $(\gamma_1,p)=(\gamma_2,p)=1$. For $k\geq 2$, define
\begin{equation}\label{certain character sums}
\mathcal{C}(m,\gamma_1,\gamma_2;p^k)
:=\frac{1}{p^{k/2}}\mathop{\sideset{}{^*}\sum
\sideset{}{^*}\sum}_{\substack{x_1, x_2\,\bmod p^k
\\ \gamma_1\overline{x_1}-\gamma_2 \overline{x_2}+m\equiv 0\bmod p^k}}
\Kl_3(x_1; p^k)\overline{\Kl_3(x_2; p^k)}.
\end{equation}

An estimate for $\mathcal{C}(m,\gamma_1,\gamma_2;p^k)$ can be obtained in an elementary manner by reducing the problem to a set of congruence conditions.
\begin{lemma}\label{2.5}
For a $p$-adic integer $\alpha\in \mathbb{Z}_p$, denote its $p$-adic order as $\nu_p (\alpha)$.
\begin{enumerate}
 \item
 If $m=0$, $\mathcal{C}(m,\gamma_1,\gamma_2;p^k)$ vanishes unless
$\gamma_1\equiv \gamma_2\bmod p^{\lfloor k/2\rfloor}$, in this case we have
$$
\mathcal{C}(m,\gamma_1,\gamma_2;p^k)\ll
p^{\lceil k/2\rceil}.
$$
\item
If $m\neq 0$, we have
$$
\mathcal{C}(m,\gamma_1,\gamma_2;p^k)\ll
p^{\min\{\nu_p(m),\lceil k/2\rceil \}}.
$$
\end{enumerate}
\end{lemma}
\begin{proof}
We perform some initial transformation. First suppose that $k$ is even. Then for $i=1,2$, we can write
\begin{equation}\label{xidecomp0}
x_i=p^{k/2}a_i+b_i,\,\,1\leq a_i,b_i\leq p^{k/2},\,\,(b_i,p)=1.
\end{equation}
When $k$ is odd, we replace $p^{k/2}\mapsto p^{(k+1)/2}$ in \eqref{xidecomp0}
and proceed the same as above.
From \eqref{xidecomp0} we obtain
\begin{equation}\label{splitting}
\overline{x_i}\equiv\overline{b_i}-p^{k/2}\overline{b_i}^2a_i \bmod {p^{k}}.
\end{equation}
Plugging \eqref{splitting}
we see that the congruence
\begin{equation*}
\gamma_1\overline{x_1}-\gamma_2\overline{x_2}+m=0 \bmod {p^k},
\end{equation*}is equivalent to
\begin{equation}\label{beta2alpha2}
\begin{aligned}
&\overline{b_2} \equiv \gamma_1\overline{\gamma_2}\overline{b_1}+\overline{\gamma_2}m \bmod {p^{k/2}},\\
&a_2 \equiv \gamma_1\overline{\gamma_2}b_2^2\overline{b_1}^2\alpha_1+h(b_1) \bmod {p^{k/2}},
\end{aligned}
\end{equation}
where
\begin{equation*}
h(b_1)=\overline{\gamma_2}b_2^2\cdot\frac{(\gamma_1\overline{b_1}
-\gamma_2\overline{b_2}+m)}{p^{k/2}}.
\end{equation*}
In this case, we will use Lemma \ref{lem:Kloosterman1}. Hence, the hyper-Kloosterman sum in \eqref{certain character sums} vanishes unless we have $x_i\in \mathbb{Z}_p^{\times3}$. Suppose $r_i$ is a solution to $r_i^3 \equiv x_i \bmod p^k$. After a power series expansion, we obtain
$$
r_i \equiv b_i^{1/3}+3^{-1}b_i^{-2/3}a_ip^{k/2} \bmod p^k.
$$
Substituting this expansion into \eqref{certain character sums}, we see that
\begin{align*}
\mathcal{C}(m,\gamma_1,\gamma_2;p^k)\ll \frac{1}{p^{k/2}}& \Bigg|\;
\sideset{}{^*}\sum_{b_1\bmod p^{k/2}}
e\left(\frac{b_1^{1/3}-b_2^{1/3}+h(b_1)}{p^{k}}\right)\\
&\times\sum_{ a_1\bmod p^{k/2}}
e\left(\frac{b_1^{-2/3}-\gamma_1\overline{\gamma_2}b_2^{4/3}
\overline{b_1}^2}{p^{k/2}}a_1\right)\Bigg|.
\end{align*}
Executing the linear sum over $a_1$ and taking absolute values, we obtain the bound
\begin{equation*}
\mathcal{C}(m,\gamma_1,\gamma_2;p^k)\ll
\;\sideset{}{^*}\sum_{\substack{b_1\bmod p^{k/2}\\
b_1^{4/3}\gamma_2\equiv b_2^{4/3}\gamma_1\bmod p^{k/2}}}1.
\end{equation*}
It remains to count the solutions to
$b_1^{4/3}\gamma_2\equiv b_2^{4/3}\gamma_1\bmod p^{k/2}$. By \eqref{beta2alpha2},
this forces
\begin{equation}\label{b1b2}
 (\gamma_1\overline{\gamma_2}\overline{b_1}
+\overline{\gamma_2}m)^{4/3}\gamma_2\equiv
\overline{b_1}^{4/3}\gamma_1\bmod p^{k/2}.
\end{equation}
Suppose $m=0$ or $\nu_p(m)\geq k/2$, then the congruence becomes
$$
\gamma_1\equiv\gamma_2\bmod p^{k/2}.
$$
In this case we use the trivial bound to get
\begin{equation*}
\mathcal{C}(m,\gamma_1,\gamma_2;p^k)\ll p^{k/2}
\delta_{\gamma_1\equiv\gamma_2\bmod p^{k/2}}.
\end{equation*}
In the case $\nu_p(m) < p^{k/2}$, we divide~\eqref{b1b2} by $p^{\nu_p(m)}$. Invoking Hensel's lemma, 
it follows that $b_1$ is determined modulo $p^{k/2-\nu_p(m)}$ and therefore
\begin{equation*}
\mathcal{C}(m,\gamma_1,\gamma_2;p^k)\ll p^{\nu_p(m)}.
\end{equation*}
This completes the proof of the lemma when $k$ is even. When $k$ is odd, we use $p^{(k+1)/2}$ instead of $p^{k/2}$ in \eqref{xidecomp0} and proceed identically. This eventually results in an extra factor of $p^{1/2}$ in the final estimate as indicated in the statement of the Lemma.
\end{proof}

\section{Proof of Theorem \ref{main-theorem2}}

We closely follow the proof presented in~\cite[Section 3]{KLM}.
Throughout $q=p^k, k\geq 2$ and $p\neq3$. We only need to prove that for $(a,q)=1$,
\begin{equation}\label{a3}
\sum_{\substack{n\geq 1 \\ n\equiv a\bmod q}}
\lambda_{1\boxplus \phi}(n)V\left(\frac{n}{X}\right)
-\frac{1}{\varphi(q)} \sum_{\substack {n\geq 1 \\ (n,q)=1}}
\lambda_{1\boxplus \phi}(n)V\left(\frac{n}{X}\right) \ll (X/q)^{1-\delta},
\end{equation}
provided that $\delta$ is sufficiently small.
By the same argument as in~\cite[Section 3]{KLM},
we can conclude that
\begin{equation}\label{AFE}
\begin{split}
\sum_{\substack{n\geq 1 \\ n\equiv a\bmod q}} \lambda_{1\boxplus \phi}(n)V\left(\frac{n}{X}\right)
=&\frac{1}{\varphi(q)} \sum_{\substack{ n\geq 1 \\ (n,q)=1}}
\lambda_{1\boxplus \phi}(n)V\left(\frac{n}{X}\right)\\&+\frac{X}{q^{5/2}}
\sum_{n\geq 1}\lambda_{1\boxplus \overline{\phi}}(n)
\Kl_4(an;q)\widecheck{V}\left(\frac{n}{q^4/X}\right),
\end{split}
\end{equation}
where (see Section~\ref{2.1})
\begin{equation*}
\widecheck{V}(y)=\frac{1}{2 i\pi}\int_{(\frac{3}{2})}
\frac{L_\infty(1\boxplus\overline{\phi},s+\kappa)}
{L_\infty(1\boxplus\pi,1-s+\kappa)}\widetilde V(1-s)y^{-s}\mathrm{d}s
\end{equation*}
is a rapidly decreasing function of $y$ and $\widetilde{V}(s)$ is
the Mellin transform of $V(x)$.
Finally, combining \eqref{a3} and \eqref{AFE}, our task is neatly reduced to prove that for small $\delta>0$,
\begin{equation*}
\frac{X}{q^{5/2}}
\sum_{n\geq 1}\lambda_{1\boxplus \overline{\phi}}(n)
\Kl_4(an;q)
\widecheck{V}\left(\frac{n}{q^4/X}\right)
\ll (X/q)^{1-\delta}.
\end{equation*}

Plugging in the definition
$$\lambda_{1 \boxplus\overline{\phi}}(n)
=\sum_{\ell m=n}A(1,m),$$
we have
\begin{equation*}
\begin{split}
\frac{X}{q^{5/2}}
\sum_{n\geq 1}\lambda_{1\boxplus \overline{\phi}}(n)
\Kl_4(an;q)
\widecheck{V}\left(\frac{n}{q^4/X}\right)=
\frac{X}{q^{5/2}}\mathop{\sum\sum}_{\ell,m\geq 1}A(1,m)
\Kl_4(a\ell m;q)\widecheck{V}\left(\frac{\ell m}{q^4/X}\right).
\end{split}
\end{equation*}
By applying smooth dyadic partitions of the $\ell$-sum and the $m$-sum, we are reduced to considering sums of the form
\begin{equation}\label{LMsums}
\mathcal{B}(L,M):=\frac{X}{q^{5/2}}\sum_{\ell\geq 1}\sum_{m\geq 1}
A(1,m)\Kl_4(a \ell m;q)V_1\left(\frac{\ell}{L}\right)V_2\left(\frac{m}{M}\right)
\end{equation}
for $O(\log^2 X)$ many real numbers $L,M\geq 1$ satisfying
\begin{equation}\label{LMupperbound}
LM\ll \frac{q^4}{X}.
\end{equation}
Since $|\Kl_4(a\ell m;q)|\leq 4$ and
the Rankin--Selberg bound \eqref{GL3-Rankin--Selberg},
the trivial bound for $\mathcal{B}(L, M)$ is
\begin{equation*}
\frac{X^{1+\varepsilon}}q\bigg(\frac{LM}{q^{3/2}}\bigg),
\end{equation*}
which is good enough if $q\leq X^{2/5-\varepsilon}$,
and henceforth we assume that $q\geq X^{2/5-\varepsilon}$.
In particular, we may assume that $LM\geq q^{3/2-\varepsilon}$
for some fixed $\varepsilon>0$ that can be chosen as small as necessary.

To obtain nontrivial cancellation for the sum
\eqref{LMsums}, we split the argument into several cases.
The treatment of these depends on the relative sizes of $L$ and $M$.

\subsection{The case $M\geq q^{7/5}$}

If $L$ is small,
we apply Theorem \ref{Bound theorem} directly, getting
\begin{align*}
\mathcal{B}(L,M)&\ll_{\overline{\phi}} \frac{X^{1+\varepsilon}}{q^{5/2}}L
(p^{3/4}M^{3/4}q^{3/10}+M^{1/2}q^{13/20})\\
&\ll p^{3/4}\frac{X^{1+\varepsilon}}{q}\frac{L^{1/4}q^{9/5}}{X^{3/4}},
\end{align*}
provided $LM\ll q^4/X$.
In particular this bound is suitable as long as
$$q\leq X^{5/12-\varepsilon}L^{-5/36}$$
for some fixed $\varepsilon>0$.
In particular, since $L\geq 1$, this implies that $q\leq X^{5/12}$.
In view of this and \eqref{LMupperbound} we may assume that
$$LM\leq q^{4-12/5}=q^{8/5},$$
which implies (since we have assumed $M\geq q^{7/5}$)
that $L\leq q^{1/5}$.

\subsection{The case $L\geq q^{1/2}$}
 In that situation we can improve over the trivial bound by
 applying the Poisson summation formula (Lemma~\ref{Fourier transform}) in the $\ell$ variable, i.e.,
 \begin{align*}
\mathcal{B}(L,M)
&=\frac{X}{q^{5/2}}\sum_{m\geq 1}A(1, m)
\left(\frac{L}{q^{1/2}}\sum_{\ell \in \mathbb{Z}}
\Kl_3(am\overline{\ell};q)
\widehat{V_1}\left(\frac{\ell}{q/L}\right)\right)
V_2\left(\frac{m}{M}\right)\notag\\
&\ll\frac{X^{1+\varepsilon}}{q}\frac{M}{q}.
 \end{align*}
Here we have used $(\ell,q)=1$ (Lemma \ref{lem:Kloosterman=0-initial}).
This bound is good as long as
$M\leq q^{1-\varepsilon}$,
which occurs as soon as
\begin{equation}
    \label{goodcase3}
L\geq q^{3/5+\varepsilon}.
\end{equation}

\subsection{The case $L\leq q^{1-\varepsilon}$}

By the Cauchy--Schwarz inequality, we have
\begin{align}\label{aa}
\mathcal{B}(L,M) &\ll\frac{X}{q^{5/2}}\bigg(\sum_{m\geq 1}
|A(1, m)|^2V_2\left(\frac{m}{M}\right)\bigg)^{1/2}
\biggl(\sum_{m\geq 1}\bigg|\sum_{\ell\geq 1}\Kl_4(a\ell m;q)
V_1\left(\frac{\ell}{L}\right)\bigg|^2V_2\left(\frac{m}{M}\right)\biggr)^{1/2}
\notag \\ &\ll\frac{XM^{1/2}}{q^{5/2}}
\bigg(\sum_{\ell_1,\ell_2\geq 1}V_1\left(\frac{\ell_1}{L}\right)
\overline{V_1\left(\frac{\ell_2}{L}\right)}\sum_{m\geq 1}
\Kl_4(a\ell_1 m;q)\overline{\Kl_4(a\ell_2 m;q)}
V_2\left(\frac{m}{M}\right)\bigg)^{1/2}.
\end{align}
By applying the Poisson summation formula in the innermost $m$-sum, we get
\begin{equation*}
\sum_{m\geq 1}
\Kl_4(a\ell_1 m;q)\overline{\Kl_4(a\ell_2 m;q)}
V_2\left(\frac{m}{M}\right)=\frac{M}{q^{1/2}}\sum_{m\in \mathbb{Z}}\mathcal{C}_{a}(m,\ell_1,\ell_2;q)
\widehat{V_2}\left(\frac{m}{q/M}\right),
\end{equation*}
where
$$\mathcal{C}_{a}(m,\ell_1,l_2;q):=\frac{1}{q^{1/2}}
\sum_{\gamma\in \mathbb{Z}/q\mathbb{Z}}
\Kl_4(a\ell_1 \gamma;q)\overline{\Kl_4(a\ell_2 \gamma;q)}
e\left(\frac{\gamma m}{q}\right).$$
We write
$$
\Kl_4(a;q)=\frac{1}{q^{1/2}}\,\;\sideset{}{^*}
\sum_{x\bmod q}e\left(\frac{ax}{q}\right)
\Kl_3(\overline{x}; q).
$$
Thus
\begin{align*}
 \mathcal{C}_{a}(m,\ell_1,\ell_2;q)&=\frac{1}{q^{3/2}}
\mathop{\sideset{}{^*}\sum\sideset{}{^*}\sum}_{\substack{x,y\,\bmod q}}
\Kl_3(\overline{x}; q)\overline{\Kl_3(\overline{y}; q)}
\sum_{\gamma\in \mathbb{Z}/q\mathbb{Z}}
e\left(\frac{a\ell_1x-a\ell_2 y+m}{q}\gamma\right)\\
&=\frac{1}{q^{1/2}}\mathop{\sideset{}{^*}\sum\sideset{}{^*}\sum}_
{\substack{x,y\,\bmod q \\ a\ell_1x-a\ell_2 y+m\equiv 0 \bmod q}}
\Kl_3(\overline{x}; q)\overline{\Kl_3(\overline{y}; q)}.
\end{align*}
By Lemma~\ref{2.5}, we have the following bound
\begin{align}\label{bb}
&\frac{M}{q^{1/2}}\sum_{\ell_1,\ell_2\geq 1}V_1\left(\frac{\ell_1}{L}\right)
\overline{V_1\left(\frac{\ell_2}{L}\right)}
\sum_{m\in \mathbb{Z}}
\mathcal{C}_{a}(m,\ell_1,\ell_2;q)
\widehat{V_2}\left(\frac{m}{q/M}\right)\nonumber\\
&\ll \frac{M}{q^{1/2}}
\left(Lp^{\lceil k/2\rceil}+\frac{L^2q}{M}\right)\nonumber\\
&\ll p^{1/2}ML+L^2q^{1/2}.
\end{align}
Combining \eqref{aa} and \eqref{bb},
we obtain that the original sum can be bounded as follows
\begin{align}\label{MNbound5}
\mathcal{B}(L,M)\ll
\frac{X^{1+\varepsilon}}q\left(p^{1/2}\frac{1}{L}\frac{q^5}{X^2}
+L\frac{q^{3/2}}{X}\right)^{1/2}.
\end{align}
In view of \eqref{goodcase3} we will apply this bound only when
$L\leq q^{3/5+\varepsilon}$ for $\varepsilon>0$ small enough
(in particular so that $q^{3/5+\varepsilon}\leq q^{ 1-\varepsilon}$).
Assuming this the second term in the parentheses
on the right hand side of  \eqref{MNbound5} satisfies
$$L\frac{q^{3/2}}{X}\leq X^{-5/24+\varepsilon}.$$
Therefore, under these conditions \eqref{MNbound5} is good as soon as
\begin{equation*}
  q\leq X^{2/5-\varepsilon}L^{1/5}.
\end{equation*}

\subsection{Conclusion}

Let $L_0=X^{3/61}$ be the solution of the equation
$$X^{5/12}L_0^{-5/36}=X^{2/5}L_0^{1/5}=X^{25/61}
=X^{2/5+3/305}.$$

We need to show that for any  small enough  $\varepsilon>0$
and any prime power $q=p^k$ satisfying
$$X^{2/5-\varepsilon}\leq q\leq X^{25/61-\varepsilon},$$ one has
\begin{equation*}
\sum_{\substack{ n\geq 1 \\ n\equiv a\bmod q}}
\lambda_{1\boxplus\phi}(n)V\left(\frac{n}{X}\right)
-\frac{1}{\varphi(q)} \sum_{\substack{ n\geq 1 \\ (n,q)=1}}
\lambda_{1\boxplus\phi}(n)V\left(\frac{n}{X}\right)\ll p^{3/4}(X/q)^{1-\delta},
\end{equation*}
for some $\delta=\delta(\varepsilon)>0$.
Verbatim the same as in~\cite[Section 3.4]{KLM}, we obtain that this bound holds for any of the sums~\eqref{LMsums} for $L,M$ satisfying $1\leq LM\leq q^4/X.$ This completes the proof of Theorem \ref{main-theorem2}.

\section{Proof of Theorem \ref{main-theorem1}}

In this section we deduce Theorem \ref{main-theorem1} from Theorem \ref{main-theorem2}. 
Our analysis is similar to the one presented in~\cite[Section 2]{KLM}.

Set $\phi=\sym^2 f$. By \eqref{convident}, we have
\begin{equation*}
    \lambda_f(n)^2=\sum_{d^2r=n}\mu(d) \lambda_{1\boxplus{\phi}}(r).
\end{equation*}
Moreover, we have
\begin{align*}
  &\sum_{\substack{n\geq 1\\ n\equiv a\bmod q}}\lambda_f(n)^2 V\left(\frac{n}{X}\right)-\frac{1}{\varphi(q)}\sum_{\substack{n\geq 1\\ (n,q)=1}}\lambda_f(n)^2V\left(\frac{n}{X}\right)\\
   &=\sum_{\substack{n\geq 1\\ n\equiv a\bmod q}}\sum_{d^2r=n}\mu(d) \lambda_{1\boxplus{\phi}}(r)V\left(\frac{n}{X}\right)-\frac{1}{\varphi(q)}
   \sum_{\substack{n\geq 1\\ (n,q)=1}}\sum_{d^2r=n}\mu(d) \lambda_{1\boxplus{\phi}}(r)V\left(\frac{n}{X}\right).
\end{align*}
Let $a'=a'(a,d,q)$ be any integer such that $d^2a'\equiv a \bmod q.$
Since $(a,q)=(d,q)=1$, then the last congruence becomes $r\equiv a' \bmod q$.
Therefore, we have
\begin{align*}
  &\sum_{\substack{n\geq 1\\ n\equiv a\bmod q}}\lambda_f(n)^2V\left(\frac{n}{X}\right)-\frac{1}{\varphi(q)}\sum_{\substack{n\geq 1\\ (n,q)=1}}\lambda_f(n)^2V\left(\frac{n}{X}\right)\\
   &=\sum_{(d,q)=1}\mu(d)\left(\sum_{\substack{r\geq 1\\ r\equiv a'\bmod q}} \lambda_{1\boxplus{\phi}}(r)V\left(\frac{r}{X/d^2}\right)-\frac{1}{\varphi(q)}
   \sum_{\substack{r\geq 1\\ (r,q)=1}}\lambda_{1\boxplus{\phi}}(r)
   V\left(\frac{r}{X/d^2}\right)\right).
\end{align*}
Let $D=X^{\varpi'}$ for some $\varpi'>0$ small enough (we will choose $\varpi'>0$ later) so that
\begin{equation}\label{condition-on-q}
   q\leq (X/D^2)^{25/61-\varepsilon/10}.
\end{equation}

By applying Theorem \ref{main-theorem2}, we obtain
\begin{align*}
 \sum_{d\leq D}\mu(d)&\left(\sum_{\substack{r\geq 1\\ r\equiv a'\bmod q}} \lambda_{1\boxplus{\phi}}(r)V\left(\frac{r}{X/d^2}\right)-\frac{1}{\varphi(q)}
   \sum_{\substack{r\geq 1\\ (r,q)=1}}\lambda_{1\boxplus{\phi}}(r)
   V\left(\frac{r}{X/d^2}\right)\right)\\
   &\ll \sum_{d\leq D}p^{3/4}\left(\frac{X}{d^2q}\right)^{1-\delta}\ll p^{3/4}(X/q)^{1-\delta}
\end{align*}
holds for $\delta<1/2$.

By using $\lambda_{1\boxplus{\phi}}(n)=\sum_{r|n}\lambda_{\sym^2 f}(r)$ and $|\lambda_{\sym^2 f}(n)|\ll n^{2\varpi_f+\varepsilon}$ with $\varpi_f$ satisfying \eqref{RPbound}, we have the trivial bound
\begin{align*}
\sum_{\substack{r\geq 1\\ r\equiv a'\bmod q}} & \lambda_{1\boxplus{\phi}}(r)V\left(\frac{r}{X/d^2}\right)-\frac{1}{\varphi(q)}
   \sum_{\substack{r\geq 1\\ (r,q)=1}}\lambda_{1\boxplus{\phi}}(r)
   V\left(\frac{r}{X/d^2}\right)\\
   &\ll \frac{(X/d^2)}{q}\max_{r\ll X/d^2}|\lambda_{1\boxplus\phi}(r)|\ll X^\varepsilon\frac{(X/d^2)(X/d^2)^{2\varpi_f}}{q}.
\end{align*}
Since $$q\leq X^{\frac{25}{61(1+4\varpi_f)}-\varepsilon},$$
we see that \eqref{condition-on-q} is satisfied
for
$$D= X^{\frac{2\varpi_f}{1+4\varpi_f}+\varepsilon}.$$
We obtain
\begin{align*}
 \sum_{d > D}\mu(d)&\left(\sum_{\substack{r\geq 1\\ r\equiv a'\bmod q}} \lambda_{1\boxplus{\phi}}(r)V\left(\frac{r}{X/d^2}\right)-\frac{1}{\varphi(q)}
   \sum_{\substack{r\geq 1\\ (r,q)=1}}\lambda_{1\boxplus{\phi}}(r)
   V\left(\frac{r}{X/d^2}\right)\right)\\
   &\ll \frac{X^{1+o(1)}(X/D^2)^{2\theta_f}}{Dq}\ll (X/q)^{1-\delta'}
\end{align*}
for some $\delta'=\delta'(\varepsilon)>0$.

Putting the two bounds together we conclude that
\begin{align*}
 \sum_{\substack{n\geq 1\\ n\equiv a\bmod q}}\lambda_f(n)^2V\left(\frac{n}{X}\right)-\frac{1}{\varphi(q)}\sum_{\substack{n\geq 1\\ (n,q)=1}}\lambda_f(n)^2V\left(\frac{n}{X}\right) \ll p^{3/4}(X/q)^{1-\min(\delta,\delta')}.
\end{align*}
This completes the proof of Theorem \ref{main-theorem1}.

\section{Proof of Theorem \ref{Bound theorem}}

In this section, we provide the details of the proof for Theorem \ref{Bound theorem}.
By applying Lemma \ref{lem:Kloosterman=0-initial}, we can assume $(\ell, p)=1$ 
since otherwise the trace function vanishes.

\subsection{Applying the delta method}

In order to use the delta symbol method in Lemma \ref{DFI's},
we proceed as \cite[Eq. (3.7)]{LMS}
(see \cite[Lemma 3.2 ]{WZ} for more details) by writing $\delta_{n=0}$
in a more analytic form:
\begin{equation}\label{circle method}
 \begin{split}
   \delta_{n=0}
=&\sum_{r=0}^{\lambda}\frac{1}{C}\sum_{\substack{c\leq C\\ (c,p)=1}}\frac{1}{cp^\lambda}
\;\sideset{}{^*}\sum_{u\bmod cp^{\lambda-r}}
e\left(\frac{un}{cp^{\lambda-r}}\right)
\int_{\mathbb{R}}
  g(c,\zeta)e\left(\frac{n\zeta}{cCp^\lambda}\right)\mathrm{d}\zeta\\
&+\sum_{s=1}^{[\log C/\log p]}
\frac{1}{C}\sum_{\substack{c\leq C/p^s\\ (c,p)=1}}\frac{1}{cp^{\lambda+s}}
\sideset{}{^*}\sum_{u\bmod cp^{\lambda+s}}
e\left(\frac{un}{cp^{\lambda+s}}\right)
\int_{\mathbb{R}}
  g(p^sc,\zeta)e\left(\frac{n\zeta}{cCp^{\lambda+s}}\right)
  \mathrm{d}\zeta,
 \end{split}
\end{equation}
where $\lambda \in \mathbb{Z}$.
Instead of using the entire modulus $q$ for the conductor lowering mechanism, 
we only use a part $p^\lambda$, where $\lambda < k$ is chosen optimally later. 
This introduces more terms in the diagonal while having less impact in the off-diagonal.

Now we write
$$
S(N)=\sum_{n\geq 1}A(1, n)W\left(\frac{n}{N}\right)\\
\sum_{\substack{m \geq 1 \\ p^{\lambda}\mid n-m}}\Kl_4(m\ell;q)
V\left(\frac{m}{N}\right)\delta_{n=m},
$$
with a compactly supported smooth function
$W$ such that $\rm supp$ $W \subset [1,2]$ and $W^{(j)}\ll 1$ for $j\geq 1$.
Applying \eqref{circle method} with
$ C=\sqrt{N/p^{\lambda}}, $
we have
\begin{align*}
S(N)=\,&\sum_{n\geq1}A(1, n)W\left(\frac{n}{N}\right)
\sum_{m\geq1}\Kl_4(m\ell;q)V\left(\frac{m}{N}\right)\\
&\times\bigg\{\sum_{r=0}^{\lambda}\frac{1}{C}
\sum_{\substack{c\leq C\\ (c,p)=1}}\frac{1}{cp^\lambda}
\;\sideset{}{^*}\sum_{u\bmod cp^{\lambda-r}}
e\left(\frac{u(n-m)}{cp^{\lambda-r}}\right)
\int_{\mathbb{R}}
g(c,\zeta)e\left(\frac{(n-m)\zeta}{cCp^\lambda}\right)\mathrm{d}\zeta\\
&+\sum_{s=1}^{[\log C/\log p]}
\frac{1}{C}\sum_{\substack{c\leq C/p^s\\ (c,p)=1}}\frac{1}{cp^{\lambda+s}}
\sideset{}{^*}\sum_{u\bmod cp^{\lambda+s}}
e\left(\frac{u(n-m)}{cp^{\lambda+s}}\right)
\int_{\mathbb{R}}
  g(p^sc,\zeta)e\left(\frac{(n-m)\zeta}{cCp^{\lambda+s}}\right)
  \mathrm{d}\zeta\bigg\}.
\end{align*}
In the above sum, we only consider the first term in the braces with $r = 0$, that is,
\begin{align*}
\widetilde{S}(N):=&\;\sum_{n\geq1}A(1, n)W\left(\frac{n}{N}\right)
\sum_{m\geq1}\Kl_4(m\ell;q)V\left(\frac{m}{N}\right)\\
&\times\frac{1}{C}\sum_{1\leq c\leq C} \;
\frac{1}{cp^\lambda}\;\sideset{}{^*}
\sum_{u\bmod{cp^\lambda}}
e\left(\frac{(n-m)u}{cp^\lambda}\right)
\int_\mathbb{R}g(c,\zeta) e\left(\frac{(n-m)\zeta}
{cCp^\lambda}\right)\mathrm{d}\zeta.
\end{align*}
\begin{remark}\label{remark5.1}
The other terms are lower order terms, which can be treated similarly.
The same method works for the other sums and will
give better bounds as the lengths of those sums are shorter.
c.f.~\cite[Section 3]{HX} and~\cite[Section 2.1]{SY}.
\end{remark}
Interchanging the order of integration and summations, we get
\begin{align*}
\widetilde{S}(N)=\frac{1}{C}&\sum_{\substack{1\leq c\leq C \\ (c,p)=1}}\frac{1}{cp^\lambda}
\;\sideset{}{^*}\sum_{u\bmod{cp^\lambda}}
\int_\mathbb{R}g(c, \zeta)\sum_{n\geq1}A(1, n)
e\left(\frac{un}{cp^{\lambda}}\right)
W\left(\frac{n}{N}\right)e\left(\frac{n\zeta}{cCp^\lambda}\right)\\
&\times\sum\limits_{m\geq1}\Kl_4(m\ell;q)
e\left(-\frac{um}{cp^{\lambda}}\right)
V\left(\frac{m}{N}\right)e\left(-\frac{m\zeta}{cCp^\lambda}\right)
\mathrm{d}\zeta.
\end{align*}

\subsection{Application of summation formulas}
We now proceed to estimate $\widetilde{S}(N)$.
We first consider the sum over $m$.
Applying Poisson summation with
modulus $cp^k$ on the $m$-sum, we get
\begin{equation}\label{after poisson}
  \begin{split}
\sum_{\beta \bmod cp^k}&\Kl_4(\beta \ell; q)e\left(-\frac{u\beta}{cp^\lambda}\right)
\sum_{m\equiv\beta \bmod cp^k}V\left(\frac{m}{N}\right)
e\left(-\frac{m\zeta}{cCp^\lambda}\right) \\
&=\frac{N}{cp^k}\sum_{m\in \mathbb{Z}}\sum_{\beta \bmod cp^k}\Kl_4(\beta \ell; q)e\left(-\frac{m+up^{k-\lambda}}{cp^k}\beta\right)
\mathfrak{I}(m,c,\zeta),
  \end{split}
\end{equation}
where
\begin{equation}\label{integral I}
\mathfrak{I}(m,c,\zeta)=\int_{\mathbb{R}}V(y)
e\left(-\frac{N\zeta y}{cCp^\lambda}
+\frac{mNy}{cp^k} \right)\mathrm{d}y.
\end{equation}
By repeated integration by parts, the integration is negligibly
small if $mN/cp^k \gg N^{1+\varepsilon}/cCp^\lambda$. Thus we only need to consider
the range
$1\leq |m| \leq N^\varepsilon p^k/Cp^\lambda$.

Since $(c,p)=1$, the $\beta$-sum
factors as
\begin{equation*}
 \sum_{\beta \bmod cp^k}\Kl_4(\beta \ell;q)
e\left(-\frac{m+up^{k-\lambda}}{cp^k}\beta\right)
=c p^{k/2}\Kl_3\big(c \ell\overline{(m+up^{k-\lambda})};q\big)
\times \delta_{m\equiv -up^{k-\lambda} \bmod c}.
\end{equation*}
Now we consider the $n$-sum. Applying the $\GL_3$ Voronoi formula (Lemma~\ref{voronoiGL3}) with $w(y)=W(y/N)e(\zeta y/Ccp^\lambda)$,
we transform the $n$-sum into
\begin{equation}\label{$n$-sum after GL3 Voronoi}
cp^\lambda\sum\limits_{\pm}\sum_{n_1|cp^\lambda}\sum_{n_2}
\frac{A(n_2,n_1)}{n_1n_2}S\left(\overline{u},\pm n_2;\frac{cp^\lambda}{n_1}\right)
\mathcal{W}^{\pm}\left(n_1^2 n_2, c, \zeta\right),
\end{equation}
where by \eqref{intgeral transform-3},
\begin{equation*}
\mathcal{W}^{\pm}\left(n_1^2 n_2, c, \zeta\right)
=\frac{1}{2\pi i}\int_{(\sigma)}\left(\frac{Nn_1^2 n_2}{c^3p^{3\lambda}}\right)^{-s}
\gamma_{\pm}(s)
W^{\dag}\left(\frac{N\zeta}{cCp^\lambda},-s\right)\mathrm{d}s
\end{equation*}
with
\begin{equation*}
 W^{\dag}(\xi,s)
 =\int_{0}^{\infty}W(y)e(\xi y)y^{s-1} \mathrm{d} y.
\end{equation*}
By Stirling's formula, for $\sigma\geq -1/2$,
\begin{equation}\label{A bound}
  \gamma_{\pm}(\sigma+i\tau)\ll_{\pi,\sigma}(1+|\tau|)^{3\left(\sigma+1/2\right)}.
\end{equation}
By \cite[Lemma 5]{Mun1}, we have
\begin{equation*}
\begin{split}
W^{\dag}\left(\frac{N\zeta}{cCp^\lambda},-s\right)\ll_j&
\min \left\{ \left(\frac{1+|\mathrm{Im}(s)|}{N|\zeta|/cCp^\lambda} \right)^j ,
\left(\frac{1+N|\zeta|/cCp^\lambda}{|\mathrm{Im}(s)|} \right)^j \right\}\\
\ll_j& \min \left\{1, \left(\frac{N^{1+\varepsilon}}{cCp^\lambda|\mathrm{Im}(s)|} \right)^j \right\}.
\end{split}
\end{equation*}
This together with \eqref{A bound} implies that
\begin{align*}
\mathcal{W}^{\pm}\left(n_1^2 n_2, c, \zeta\right)&\ll_j \left(\frac{Nn_1^2 n_2}{c^3p^{3\lambda}}\right)^{-\sigma}
\int_{\mathbb{R}}
(1+|\tau|)^{3(\sigma+1/2)}\min\left\{1, \left(\frac{N^{1+\varepsilon}}{cCp^\lambda|\mathrm{Im}(s)|} \right)^j \right\}\mathrm{d}\tau\\
&\ll\left(\frac{N^{1+\varepsilon}}{cCp^\lambda}\right)^{5/2}\left(\frac{C^3n_1^2n_2}
{N^{2+\varepsilon}}\right)^{-\sigma},
\end{align*}
by choosing $j=3\sigma+5/2$ and by noting that $|\zeta|\leq N^{\varepsilon}$.

By taking $\sigma$ sufficiently large, one sees that
the $n_1,n_2$-sums in \eqref{$n$-sum after GL3 Voronoi} can be truncated at
$n_1^{2}n_2\leq N^{2+\varepsilon}/C^3$, at the cost of a negligible error.
After doing this truncation, we move the integration line to $\sigma=-1/2$ to get
\begin{equation*}
\mathcal{W}^{\pm}\left(n_1^2 n_2, c, \zeta\right)
=\left(\frac{Nn_1^2 n_2}{c^3p^{3\lambda}}\right)^{1/2}
\mathfrak{J}(n_1^2n_2, c, \zeta),
\end{equation*}
where
\begin{equation}\label{integral J}
\mathfrak{J}(n_1^2n_2, c, \zeta)
=\frac{1}{2\pi i}\int_{(\sigma)}\left(\frac{Nn_1^2 n_2}{c^3p^{3\lambda}}\right)^{-i\tau}
\gamma_{\pm}\left(-\frac{1}{2}+i\tau\right)
W^{\dag}\left(\frac{N\zeta}{cCp^\lambda},\frac{1}{2}-i\tau\right)\mathrm{d}\tau.
\end{equation}
Assembling the above results,
exchanging the orders of summations and integrations, and further breaking the
$(n_1,n_2)$-sums into dyadic segments $n_1^2n_2 \sim N_1$
with $1\ll N_1\ll N^{2+\varepsilon}/C^3$, we obtain
$$
\widetilde{S}(N)\ll \sum_{\pm}\sum_{\substack{1\ll N_1\ll N^{2+\varepsilon}/C^3\\ \text{dyadic}}}|\widetilde{S}(N,N_1)|,
$$
where
\begin{equation}\label{main case}
\begin{split}
\widetilde{S}(N,N_1)=\;&\frac{N^{3/2}}{Cp^{(k+3\lambda)/2}}
\sum_{\substack{1\leq c \leq C \\ (c, p)=1}}\frac{1}{c^{3/2}}
\sum_{\substack{1\leq |m| \leq N^\varepsilon p^k/Cp^\lambda\\ (m,c)=1}}
\sum_{\substack{n_1^2n_2\sim N_1\\ n_1|cp^\lambda}}
\frac{A(n_2,n_1)}{n_2^{1/2}}\\
&\times\mathcal{J}(n_{1}^2n_2,m,c)\;\mathcal{C}(m,c,n_1,n_2),
\end{split}
\end{equation}
with
\begin{equation}\label{integral JJ}
\mathcal{J}(n_{1}^2n_2,m,c)=\int_{\mathbb{R}}g(c,\zeta)
\mathfrak{I}(m,c,\zeta)\mathfrak{J}(n_1^2n_2, c, \zeta)
\mathrm{d}\zeta,
\end{equation}
and
\begin{equation*}
 \mathcal{C}(m,c,n_1,n_2)
 =\;\sideset{}{^*}\sum_{\substack{u\bmod{cp^\lambda}\\ m\equiv -up^{k-\lambda}\bmod c}}
\Kl_3\big(c \ell\overline{(m+up^{k-\lambda})};q\big)
S\left(\overline{u},\pm n_2;\frac{cp^\lambda}{n_1}\right).
\end{equation*}

For the sum in~\eqref{main case}, we further split it into
two sums according to $(n_1,p)=1$ or not, and write
$$\widetilde{S}(N,N_1)= \widetilde{S}_1(N,N_1)+\widetilde{S}_2(N,N_1),$$
where
\begin{equation*}
\begin{split}
\widetilde{S}_1(N,N_1)=\;&\frac{N^{3/2}}{Cp^{(k+3\lambda)/2}}
\sum_{\substack{1\leq c \leq C \\ (c, p)=1}}\frac{1}{c^{3/2}}
\sum_{\substack{1\leq |m| \leq N^\varepsilon p^k/Cp^\lambda \\ (m,c)=1}}
\sum_{\substack{n_1^2n_2\sim N_1\\ n_1|c,\ (n_1,p)=1}}\frac{A(n_2,n_1)}{n_2^{1/2}}\\
&\times\mathcal{J}(n_{1}^2n_2,m,c)\;\mathcal{C}(m,c,n_1,n_2),
\end{split}
\end{equation*}
and $\widetilde{S}_2(N,N_1)$ corresponds to the complementary sum where $p|n_1$.
\begin{remark}\label{remark5.2}
We only consider the case $(n_1,p)=1$, that is $\widetilde{S}_1(N,N_1)$,
since otherwise we have $p |n_1$ which leads to a simpler case.
More precisely, we write $n_1=n_1'n_1''$ with $n_1'|c$ and $n_1''|p^\lambda$,
then we have
\begin{equation*}
\begin{split}
\widetilde{S}_2(N,N_1)=\;&\frac{N^{3/2}}{Cp^{(k+3\lambda)/2}}
\sum_{\substack{1\leq c \leq C \\ (c, p)=1}}\frac{1}{c^{3/2}}
\sum_{\substack{1\leq |m| \leq N^\varepsilon p^k/Cp^\lambda \\ (m,c)=1}}
\sum_{n_1''|p^\lambda}\sum_{\substack{n_1'^2n_2\sim N_1/n_1''\\ n_1'|c,\ (n_1',p)=1}}\frac{A(n_2,n_1'n_2'')}{n_2^{1/2}}\\
&\times\mathcal{J}(n_1'^2n_1''^2n_2,m,c)\;\mathcal{C}(m,c,n_1'n_1'',n_2).
\end{split}
\end{equation*}
The contribution of $\widetilde{S}_2(N,N_1)$ is actually smaller, basically because the lengths of the $n_2$-sums become shorter. In fact, the strategy we use for $\widetilde{S}_1(N,N_1)$ still works for $\widetilde{S}_2(N,N_1)$. c.f.~\cite[Section 4]{HX} and~\cite[Section 2.2]{SY}.
\end{remark}
Before further analysis, we make a computation of the character sums $\mathcal{C}(m,c,n_1,n_2)$.
Writing
$$
S\left(\overline{u},\pm n_2;\frac{cp^{\lambda}}{n_1}\right)
=S\left(\overline{u}\overline{p^{\lambda}},\pm n_2\overline{p^{\lambda}};\frac{c}{n_1}\right)
S\left(\overline{u}\overline{c/n_1},\pm n_2\overline{c/n_1};p^{\lambda}\right),
$$
which follows from the Chinese remainder theorem. Then we obtain
\begin{equation}\label{character sums0}
\begin{split}
 \mathcal{C}(m,c,n_1,n_2)
=&\;S\left(-\overline{m}p^{k-\lambda}\overline{p^{\lambda}},\pm n_2\overline{p^{\lambda}};\frac{c}{n_1}\right)\\
&\;\times\sideset{}{^*}\sum_{u\bmod{p^\lambda}}
\Kl_3\big(c\ell\overline{(m+up^{k-\lambda})};q\big)
S\left(\overline{u}\overline{c/n_1},\pm n_2\overline{c/n_1};p^{\lambda}\right).
\end{split}
\end{equation}

\subsection{Applying the Cauchy--Schwarz inequality}

Applying the Cauchy--Schwarz inequality and using
the Rankin--Selberg estimate \eqref{GL3-Rankin--Selberg},
one sees that
$$
\widetilde{S}_1(N,N_1) \ll\frac{N^{3/2}N_1^{1/2}}{Cp^{(k+3\lambda)/2}}
\sum\limits_{\pm}
\sum\limits_{\substack{n_{1}\\ (n_1,p)=1,\ n_{1}| c}}
\mathbf{\Omega}^{1/2},
$$
with
\begin{align*}
\mathbf{\Omega}=\sum\limits_{n_2}\frac{1}{n_2}
U\left(\frac{n_{1}^2n_2}{N_1}\right)
\Bigg|\sum_{\substack{1\leq c \leq C \\ (c,p)=1}}\frac{1}{c^{3/2}}
\sum_{\substack{1\leq |m| \leq N^\varepsilon p^k/Cp^\lambda\\ (m,c)=1}}
\mathcal{J}(n_{1}^2n_2,m,c)\mathcal{C}(m,c,n_1,n_2)
\Bigg|^2.
\end{align*}
Here $U$ is a nonnegative smooth function on $(0,+\infty)$, supported on $[2/3,3]$, and such
that $U(x)=1$ for $x\in [1,2]$.

Opening the absolute square, we get
\begin{align*}
\mathbf{\Omega}=&\,
\sum_{\substack{1\leq c_1 \leq C \\ (c_1,p)=1}}\frac{1}{c_1^{3/2}}
\sum_{\substack{1\leq c_2 \leq C \\ (c_2,p)=1}}\frac{1}{c_2^{3/2}}
\sum_{\substack{1\leq |m_1| \leq N^\varepsilon p^k/Cp^\lambda\\ (m_1,c)=1}}
\sum_{\substack{1\leq |m_2| \leq N^\varepsilon p^k/Cp^\lambda\\ (m_2,c)=1}}
\sum_{n_2}\frac{1}{n_2}U\left(\frac{n_{1}^2n_2}{N_1}\right)\\
&\times\mathcal{J}(n_{1}^2n_2,m_1,c_1)\mathcal{C}(m_1,c_1,n_1,n_2)
\overline{\mathcal{J}(n_{1}^2n_2,m_2,c_2)\mathcal{C}(m_2,c_2,n_1,n_2)}.
\end{align*}
We break the $n_2$-sum into congruence classes
modulo $\widehat{c_1}\widehat{c_2}p^{\lambda}$ (where $\widehat{c_1}=c_1/n_1$ and $\widehat{c_2}=c_2/n_1$), and then apply the Poisson summation formula to the sum over $n_2$.
It is therefore sufficient to consider the following sum
\begin{equation}\label{Omega sum}
\begin{split}
 \mathbf{\Omega}=\sum_{\substack{1\leq c_1 \leq C \\ (c_1,p)=1}}\frac{1}{c_1^{3/2}}&
\sum_{\substack{1\leq c_2 \leq C \\ (c_2,p)=1}}\frac{1}{c_2^{3/2}}
\sum_{\substack{1\leq |m_1| \leq N^\varepsilon p^k/Cp^\lambda \\ (m_1,c_1)=1}} \\
  &\times\sum_{\substack{1\leq |m_2| \leq N^\varepsilon p^k/Cp^\lambda \\ (m_2,c_2)=1}} \frac{1}{\widehat{c_1}\widehat{c_2}p^{\lambda}}
\sum_{n_2\in \mathbb{Z}}|\mathfrak{C}(n_2)||\mathcal{H}(n_2)|,
\end{split}
\end{equation}
where the character sum $\mathfrak{C}(n_2):=\mathfrak{C}(n_2,m_1,m_2,c_1,c_2,n_1)$ is
given by
\begin{equation}\label{character sums}
 \mathfrak{C}(n_2)=
\sum_{\beta\bmod\widehat{c_1}\widehat{c_2}p^{\lambda}}
\mathcal{C}(m_1,c_1,n_1,\beta)
\overline{\mathcal{C}(m_2,c_2,n_1,\beta)}
e\left(\frac{n_2\beta}{\widehat{c_1}\widehat{c_2}p^{\lambda}}\right),
\end{equation}
and the integral $\mathcal{H}(n_2)=\mathcal{H}(n_2;m_1,m_2,c_1,c_2,n_1)$ is given by
\begin{equation}\label{integral-H}
 \mathcal{H}(n_2)=\int_{\mathbb{R}}
U\left(\xi\right)
\mathcal{J}\left(N_1\xi,m_1,c_1\right)
\overline{\mathcal{J}\left(N_1\xi,m_2,c_2\right)}
\, e\left(-\frac{N_1n_2\xi}{c_1c_2p^{\lambda}}\right)\frac{\mathrm{d}\xi}{\xi}.
\end{equation}

We have the following estimate for $\mathcal{H}(n_2)$.
\begin{lemma}\label{integral:lemma}
Let $\mathcal{H}(n_2)$ be defined as in \eqref{integral-H}. Then,
one has the following estimates.
\begin{enumerate}
 \item
 If $n_2 \gg \frac{N^\varepsilon C^2p^\lambda}{N_1}$, we have $\mathcal{H}(n_2) \ll N^{-A}$.
\item
If $n_2 \ll\frac{N^\varepsilon C^2p^\lambda}{N_1}$, we have $$\mathcal{H}(n_2)
\ll \frac{N^\varepsilon (c_1c_2)^{1/2}}{C}.$$
    \end{enumerate}
\end{lemma}
\begin{proof}
By \eqref{integral JJ},
\begin{equation*}
  \xi^j \frac{\partial^j }{\partial \xi^j}\mathcal{J}\left(N_1\xi,m,c\right)=
\int_{\mathbb{R}}g(c,\zeta)
\mathfrak{I}(m,q,\zeta)\xi^j \frac{\partial^j }{\partial \xi^j}\mathfrak{J}(N_1\xi, c, \zeta)
  \mathrm{d}\zeta,
\end{equation*}
where by \eqref{integral J},
\begin{align*}
\xi^j \frac{\partial^j }{\partial \xi^j}\mathfrak{J}(N_1\xi, c, \zeta)=&
\frac{1}{2\pi}
\int_{\mathbb{R}}(-i\tau)(-i\tau-1)\cdots (-i\tau-j+1) \\&\times
\left(\frac{NN_1\xi}{c^3p^{3\lambda}}\right)^{-i\tau}\gamma_{\pm}\left(-\frac{1}{2}+i\tau\right)
W^{\dag}\left(\frac{N\zeta}{cCp^\lambda},\frac{1}{2}-i\tau\right)\mathrm{d}\tau.
\end{align*}
By definition,
\begin{equation*}
W^{\dag}\left(\frac{N\zeta}{cCp^\lambda},\frac{1}{2}-i\tau\right)=
\int_0^{\infty}W(y_2)y_2^{-1/2}e\left( -\frac{\tau}{2\pi}\log y_2
+\frac{N\zeta y_2}{cCp^\lambda} \right)\mathrm{d}y_2.
\end{equation*}
By repeated integration by parts, one sees that the above integral
is negligibly small unless $ |\tau|\asymp |\zeta|N/(cCp^\lambda):=\Xi$.
Moreover, by the second derivative test for exponential integrals,
\begin{equation}\label{11}
W^{\dag}\left(\frac{N\zeta}{cCp^\lambda},\frac{1}{2}-i\tau\right)\ll (1+|\tau|)^{-1/2}.
\end{equation}
Thus
\begin{equation*}
\begin{split}
\xi^j \frac{\partial^j }{\partial \xi^j}
\mathfrak{J}(N_1\xi, c, \zeta)=\;&\frac{1}{2\pi}
\int_{\mathbb{R}}\omega\Big(\frac{|\tau|}{\Xi}\Big)
(-i\tau)(-i\tau-1)\cdots (-i\tau-j+1) \\&\times
\left(\frac{NN_1\xi}{c^3p^{3\lambda}}\right)^{-i\tau}
\gamma_{\pm}\left(-\frac{1}{2}+i\tau\right)
W^{\dag}\left(\frac{N\zeta}{cCp^\lambda},
\frac{1}{2}-i\tau\right)\mathrm{d}\tau+O(N^{-A}),
   \end{split}
\end{equation*}
where $\omega(x)\in C_c^{\infty}(0,\infty)$
satisfies $\omega^{(j)}(x)\ll_j 1$ for any integer $j\geq 0$.
Hence
\begin{equation*}
\begin{split}
\xi^j \frac{\partial^j }{\partial \xi^j}
\mathcal{J}\left(N_1\xi,m,c\right)=\;&\frac{1}{2\pi}
\int_{\mathbb{R}}g(c, \zeta)\int_{\mathbb{R}}\omega\Big(\frac{|\tau|}{\Xi}\Big)(-i\tau)(-i\tau-1)\cdots (-i\tau-j+1) \\
&\times\int_0^{\infty}W(y_2)y_2^{-1/2}e\left( -\frac{\tau}{2\pi}\log y_2
-\frac{N\zeta y_2}{cCp^\lambda} \right)\mathrm{d}y_2\\&\times
\left(\frac{NN_1\xi}{c^3p^{3\lambda}}\right)^{-i\tau}\gamma_{\pm}\left(-\frac{1}{2}+i\tau\right)
\mathfrak{I}(m,c,\zeta)\mathrm{d}\tau\mathrm{d}\zeta+O(N^{-A}),
\end{split}
\end{equation*}
where by \eqref{integral I},
\begin{equation*}
\mathfrak{I}(m,c,\zeta)=\int_{\mathbb{R}}V(y_1)
e\left(-\frac{N\zeta y_1}{cCp^\lambda}
+\frac{mNy_1}{cp^k} \right)\mathrm{d}y_1.
\end{equation*}
Next we consider the $\zeta$-integral
\begin{equation*}
\int_{\mathbb{R}}
g(c,\zeta)e\left(\frac{\zeta N (y_2-y_1)}{cCp^\lambda} \right)
\mathrm{d}\zeta.
\end{equation*}
Repeated integration by parts shows that the above integral
is negligibly small unless $ |y_1-y_2|\ll N^{\varepsilon}cCp^\lambda/N\ll N^{\varepsilon}c/C$ for $C=\sqrt{N/p^\lambda}$, where we have used \eqref{g rapid decay}.
Moreover,
by Stirling's formula (see \cite[Eq. (9)]{Mun1}),
$$
\gamma_{\pm}\left(-\frac{1}{2}+i\tau\right)=\left(\frac{|\tau|}{e\pi}\right)^{3i\tau}\Upsilon_{\pm}(\tau),
\qquad \Upsilon_{\pm}'(\tau)\ll \frac{1}{|\tau|}.
$$
Combining these estimates with
\eqref{11}, we obtain
\begin{equation*}
\xi^j \frac{\partial^j }{\partial \xi^j}\mathcal{J}\left(N_1\xi,m,c\right)
\ll \frac{cN^{\varepsilon}}{C}
\Xi^{j+1/2}\ll N^{\varepsilon} \left(\frac{C}{c} \right)^{j+1/2}.
\end{equation*}
Hence by applying integration by parts repeatedly on the $\xi$-integral in \eqref{integral-H} and evaluating the
resulting $\xi$-integral trivially, we find that
$$
\mathcal{H}(n_2)\ll_j N^\varepsilon\; \frac{(c_1c_2)^{1/2}}{C}\;
\left(\frac{N_1n_2}{C^2p^{\lambda}}\right)^{-j}
$$
for any integer $ j\geq 0$. Therefore,
$ \mathcal{H}(n_2)$ is negligibly small unless $ n_2 \leq N^\varepsilon C^2p^{\lambda}/N_1$.
Moreover, by taking $ j=0$, one has
$$
\mathcal{H}(n_2)\ll_j \frac{N^\varepsilon(c_1c_2)^{1/2}}{C}.
$$
This completes the proof of Lemma \ref{integral:lemma}.
\end{proof}

\subsection{Character sums}\label{10}

In this section we estimate the character sums in \eqref{character sums}.

By \eqref{character sums0},
we write $\beta=\widehat{c_1}\widehat{c_2}
\overline{\widehat{c_1}\widehat{c_2}}\beta_{1}+p^{\lambda}
\overline{p^{\lambda}}\beta_{2}$, with $\beta_1\bmod p^{\lambda}$ and
$\beta_2\bmod \widehat{c_1}\widehat{c_2}$, where $\widehat{c_1}=c_1/n_1$ and $\widehat{c_2}=c_2/n_1$.
Then we obtain
$$
\mathfrak{C}(n_2)=\mathfrak{C}_1(n_2)\mathfrak{C}_2(n_2),
$$
where
\begin{align*}
\mathfrak{C}_1(n_2)=\sum_{\beta\bmod \widehat{c_1}\widehat{c_2}}
   S\left(-\overline{m_1}p^{k-\lambda}\overline{p^{\lambda}},\beta\overline{p^{\lambda}};\widehat{c_1}\right)
   S\left(-\overline{m_2}p^{k-\lambda}\overline{p^{\lambda}},\beta\overline{p^{\lambda}};\widehat{c_2}\right)
   e\left(\frac{n_2\overline{p^{\lambda}}\beta}{\widehat{c_1}\widehat{c_2}}\right),
\end{align*}
and
\begin{align*}
\mathfrak{C}_2(n_2)=&\sum_{\beta\bmod{p^\lambda}}\
  \sideset{}{^*}\sum_{u_1\bmod{p^\lambda}}
  \Kl_3\big(c_1 \ell\overline{(m_1+u_1p^{k-\lambda})};q\big)
 S\left(\overline{u_1}\overline{\widehat{c_1}}, \beta\overline{\widehat{c_1}};p^{\lambda}\right)\\
 &\times\sideset{}{^*}\sum_{u_2\bmod{p^\lambda}}
 \overline{\Kl_3\big(c_2 \ell\overline{(m_2+u_2p^{k-\lambda})};q\big)}
 S\left(\overline{u_2}\overline{\widehat{c_2}}, \beta\overline{\widehat{c_2}};p^{\lambda}\right) e\left(\frac{\overline{\widehat{c_1}\widehat{c_2}}\beta n_2}
   {p^{\lambda}}\right).
\end{align*}

The following estimate for the character sum $\mathfrak{C}_1(n_2)$ was proved in \cite[Lemma 11]{Mun3}.
\begin{lemma}\label{character C1}
We have
\begin{align*}
\mathfrak{C}_1(n_2) \ll \widehat{c_1}\widehat{c_2}(\widehat{c_1},\widehat{c_2},n_2).
\end{align*}
Moreover, for $n_2=0$, the character sums vanish unless $c_1=c_2=c$ in which case
$$
\mathfrak{C}_1(0)\ll \widehat{c}^2(\widehat{c},m_1-m_2).
$$
\end{lemma}
To estimate the character sum $\mathfrak{C}_2(n_2)$, we use the
 strategy in \cite[Lemma 12]{Mun3} and \cite[Lemma 6.2]{SZ} to prove the following results.
\begin{lemma}\label{character C2}
Assume that $\lambda\leq 2k/3$. For a $p$-adic integer $\alpha\in \mathbb{Z}_p$, denote its $p$-adic order as $\nu_p (\alpha)$.
\begin{enumerate}
 \item
 For $n_2=0$, $\mathfrak{C}_2(0)$ vanishes unless
$m_1^4c_2\widehat{c_2}^6\equiv m_2^4c_1\widehat{c_1}^6 \bmod  p^{\lfloor\lambda/2\rfloor}$, in this case we have
\begin{align}\label{10.1}
\mathfrak{C}_2(n_2)\ll p^{3\lambda}.
\end{align}
\item
For $n_2\neq 0$, we have
\begin{align}\label{10.2}
\mathfrak{C}_2(n_2)\ll
p^{\lceil5\lambda/2\rceil+\min\{\nu_p(n_2), \lceil\lambda/2\rceil\}}.
\end{align}
\end{enumerate}
\end{lemma}
\begin{proof}
We write
\begin{equation*}
  (s_1, t_1)=(\overline{c_1\ell}m_1, \overline{c_1\ell}), \quad (s_2, t_2)=(\overline{c_2\ell}m_2, \overline{c_2\ell}).
\end{equation*}
Then opening the Kloosterman sums and executing the sum over $\beta$, we arrive at
\begin{equation}\label{character sum2}
\begin{split}
\mathfrak{C}_2(n_2)=&\;p^{\lambda}
\sideset{}{^*}\sum_{d\bmod {p^{\lambda}}}
\;\sideset{}{^*}\sum_{u_1\bmod  p^{\lambda}}
\;\sideset{}{^*}\sum_{u_2\bmod  p^{\lambda}}
\Kl_3\big(\overline{s_1+t_1u_1p^{k-\lambda}};q\big)\\
 &\times\overline{\Kl_3\big(\overline{s_2+t_2u_2p^{k-\lambda}};q\big)}
e\left(\frac{d\overline{u_2\widehat{c_2}}
-d\widehat{c_2}\overline{(\widehat{c_1}+d n_2)u_1\widehat{c_1}}}
{p^{\lambda}}\right).
\end{split}
\end{equation}
For $n_2 \equiv 0 \bmod p^{\lambda}$,
it follows that
\begin{align*}
\mathfrak{C}_2(n_2)=&\;p^{\lambda}
\;\sideset{}{^*}\sum_{u_1\bmod  p^{\lambda}}
\;\sideset{}{^*}\sum_{u_2\bmod  p^{\lambda}}
\Kl_3\big(\overline{s_1+t_1u_1p^{k-\lambda}};q\big)\nonumber\\
 &\times\overline{\Kl_3\big(\overline{s_2+t_2u_2p^{k-\lambda}};q\big)}
 R_{p^\lambda}(\overline{u_2\widehat{c_2}}
-\widehat{c_2}\overline{\widehat{c_1}^2u_1}),
\end{align*}
where $R_{q}(u)$ is the Ramanujan sum which is bounded by $(u,q)$. We deduce that
\begin{align*}
\mathfrak{C}_2(n_2)\ll p^{3\lambda}.
\end{align*}

Furthermore, we suppose that $\lambda$ is even. For $i=1,2$, we can write
\begin{equation}\label{aidecomp0}
u_i=p^{\lambda/2}\alpha_i+\beta_i,\quad 1\leq\alpha_i,\beta_i\leq p^{\lambda/2},\quad(\beta_i,p)=1.
\end{equation}
When $\lambda$ is odd, we replace $p^{\lambda/2}\mapsto p^{(\lambda+1)/2}$ in \eqref{aidecomp0} and proceed the same as above. From \eqref{aidecomp0} we obtain
\begin{equation}\label{splitting0}
\overline{u_i}=\overline{\beta_i}-p^{\lambda/2}
\overline{\beta_i}^2\alpha_i  \bmod {p^{\lambda}}.
\end{equation}
Using Lemma \ref{lem:Kloosterman1}, the hyper-Kloosterman sum in \eqref{character sum2} vanishes unless we have $s_i\in \mathbb{Z}_p^{\times3}$. It follows that
\begin{equation*}
  \Kl_3(\overline{t_ip^{k-\lambda}u_i+s_i}; p^{k})=\left(\frac{p^k}{3}\right)\sum_{r^3\equiv \overline{t_ip^{k-\lambda}u_i+s_i}\bmod p^k}e\left(\frac{3r}{p^k}\right),
\end{equation*}
where $\sum_{r^3\equiv \overline{t_ip^{k-\lambda}u_i+s_i}\bmod p^k}$
means summing over at most the three solutions (see Remark \ref{solution}) of the congruence $r^3\equiv \overline{t_ip^{k-\lambda}u_i+s_i}\bmod p^k$. To solve the cube congruence, we
consider $(t_ip^{k-\lambda}u_i+s_i)^{-1/3}$ in the $p$-adic field $\mathbb{Q}_p$. By Taylor expansion, we have
\begin{align*}
  &(t_1p^{k-\lambda}u_1+s_1)^{-1/3}
\equiv\sum_{j\geq0}p^{j(k-\lambda)}\theta_ju_1^{j}\ \bmod p^k, \\
 &(t_2p^{k-\lambda}u_2+s_2)^{-1/3}
\equiv\sum_{j\geq0}p^{j(k-\lambda)}\eta_ju_2^{j}\ \bmod p^k,
\end{align*}
where the coefficients $\theta_j$ and $\eta_j$ with $j\geq 1$ happen to be $p$-adic integers, since $p \ne 3$ is a prime.
Then we see that the character sum \eqref{character sum2} can be written as follows
\begin{equation}\label{character sum21}
\begin{split}
 p^{\lambda}
\;\sideset{}{^*}\sum_{u_1\bmod  p^{\lambda}}
\;\sideset{}{^*}\sum_{u_2\bmod  p^{\lambda}}&\;
e\left(\frac{\sum_{j\geq0}p^{j(k-\lambda)}
\theta_ju_1^{j}-\sum_{j\geq0}p^{j(k-\lambda)}\eta_ju_1^{j}}
{p^{k}}\right)  \\
&\qquad\qquad\times\sideset{}{^*}\sum_{d\bmod {p^{\lambda}}}e\left(\frac{d\overline{u_2\widehat{c_2}}
-d\widehat{c_2}\overline{(\widehat{c_1}+d n_2)u_1\widehat{c_1}}}
{p^{\lambda}}\right).
\end{split}
\end{equation}
Using \eqref{aidecomp0}, modulo $p^k$, we have
\begin{align*}
\sum_{j\geq0}&\;p^{j(k-\lambda)}\theta_ju_1^{j}
-\sum_{j\geq0}p^{j(k-\lambda)}\eta_ju_2^{j}\\
& =\sum_{j\geq 1}p^{j(k-\lambda)+\lambda/2}j\theta_j\beta_1^{j-1}\alpha_1-\sum_{j\geq 1}p^{j(k-\lambda)+\lambda/2}j\eta_j\beta_2^{j-1}\alpha_2+\sum_{j\geq 0}\theta_j\beta_1^j-\sum_{j\geq 0}\eta_j\beta_2^i\\
& =p^{k-\lambda/2}\theta_1\alpha_1
-p^{k-\lambda/2}\eta_1\alpha_2
+\sum_{j\geq 0}\theta_j\beta_1^j-\sum_{j\geq 0}\eta_j\beta_2^i.
\end{align*}
We have truncated the last sum up to $j=1 $ since $2(k-\lambda)+ \lambda/2\geq k$ by our assumption.
Substituting this expansion into \eqref{character sum21}
and writing $d=p^{\lambda/2} d_1+d_2$, with $d_1\bmod p^{\lambda/2}$ and
$d_2\bmod p^{\lambda/2}$, we see that
\begin{equation}\label{c1c2}
\mathfrak{C}_2(n_2)\ll
 p^{3\lambda/2}
\bigg| \mathop{\sideset{}{^*}\sum_{\beta_1\bmod{ p^{\lambda/2}}}
\;\sideset{}{^*}\sum_{\beta_2\bmod{ p^{\lambda/2}}}\,
 \;\sideset{}{^*}\sum_{d_2\bmod{p^{\lambda/2}}}}_{\substack{
 \overline{(\widehat{c_1}+n_2d_2)^2\beta_1}
    \widehat{c_2}^2n_2d_2- \overline{
   (\widehat{c_1}+n_2d_2)\beta_1}\widehat{c_2}^2
   +\widehat{c_1}\overline{\beta_2}\equiv 0\bmod {p^{\lambda/2}}}}
   h(\beta_1,\beta_2,d_2)\mathcal {T}_1\mathcal {T}_2\bigg|,
\end{equation}
where we have used \eqref{splitting0}, and
\begin{align*}
\mathcal {T}_1&=\sum_{\alpha_1\bmod{ p^{\lambda/2}}}
   e\left(\frac{\theta_1+\widehat{c_2}
   \overline{\widehat{c_1}(\widehat{c_1}+n_2d_2)\beta_1^2}
                d_2}{p^{\lambda/2}}\alpha_1\right),\\
\mathcal {T}_2&=\sum_{\alpha_2\bmod{ p^{\lambda/2}}}
   e\left(-\frac{\eta_1+\overline{\widehat{c_2}\beta_2^2} d_2 }
   {p^{\lambda/2}}\alpha_2 \right),
\end{align*}
and
\begin{align*}
h(\beta_1,\beta_2,d_2)=e\left(\frac{\sum_{j\geq 0}\theta_j\beta_1^j-\sum_{j\geq 0}\eta_j\beta_2^i}{p^{\lambda}}\right)
  e\left(\frac{\overline{\widehat{c_2}\beta_2}
  -\widehat{c_2}\overline{\widehat{c_1}(\widehat{c_1}+n_2d_2)\beta_1
  }}{p^{\lambda/2}}d_2\right).
\end{align*}
Thus, $\mathcal {T}_1$ vanishes unless $\theta_1+\widehat{c_2}
   \overline{\widehat{c_1}(\widehat{c_1}+n_2d_2)\beta_1^2}
                d_2\equiv 0\bmod p^{\lambda/2}$.
Similarly,
\begin{align*}
\mathcal {T}_2=p^{\lambda/2}\delta_{\eta_1+\overline{\widehat{c_2}\beta_2^2}d_2
   \equiv 0 \bmod p^{\lambda/2}}.
\end{align*}
Plugging these into \eqref{c1c2} we obtain
\begin{equation}\label{ff}
\mathfrak{C}_2 \ll p^{5\lambda/2}
\mathop{\sideset{}{^*}\sum_{\beta_1\bmod{ p^{\lambda/2}}}
\;\sideset{}{^*}\sum_{\beta_2\bmod{ p^{\lambda/2}}}\,
 \;\sideset{}{^*}\sum_{d_2\bmod{p^{\lambda/2}}}} 
_{\substack{\overline{(\widehat{c_1}+n_2d_2)^2\beta_1}
    \widehat{c_2}^2n_2d_2- \overline{
   (\widehat{c_1}+n_2d_2)\beta_1}\widehat{c_2}^2
   +\widehat{c_1}\overline{\beta_2}\equiv 0\bmod{p^{\lambda/2}}\\
   \theta_1+\widehat{c_2}
   \overline{\widehat{c_1}(\widehat{c_1}+n_2d_2)\beta_1^2}
                d_2\equiv 0\bmod p^{\lambda/2}\\
   \eta_1+\overline{\widehat{c_2}\beta_2^2}d_2
   \equiv 0 \bmod p^{\lambda/2}}} 1.
\end{equation}

To count the numbers of $\beta_1,\beta_2$ and $d_2$, we solve the three congruence
equations in \eqref{ff}.

(1) If $n_2=0$ or $\nu_p(n_2)\geq \lambda/2$, we have
\begin{align*}
\left\{\begin{array}{l}
\beta_2\equiv
\overline{\widehat{c_2}^2}\widehat{c_1}^2\beta_1
  \bmod p^{\lambda/2},\\
  d_2\equiv -\theta_1\widehat{c_1}^2\overline{\widehat{c_2}}
                \beta_1^2 \bmod p^{\lambda/2},\\
   d_2\equiv -\eta_1\widehat{c_2}\beta_2^2
   \bmod p^{\lambda/2}.
\end{array}\right.
\end{align*}
By the last two equations, one sees that $\mathfrak{C}_2(n_2)$ vanishes unless
$\theta_1\widehat{c_2}^2\equiv \eta_1\widehat{c_1}^2\bmod p^{\lambda/2}$.
This forces
\begin{equation*}
m_1^4c_2\widehat{c_2}^6\equiv m_2^4c_1\widehat{c_1}^6 \bmod  p^{\lambda/2}.
\end{equation*}
Moreover, for fixed $\beta_1,\beta_2$ and $d_2$ are uniquely determined
modulo $p^{\lambda/2}$. Therefore,
\begin{align}\label{6.7}
\mathfrak{C}_2(n_2)
\ll p^{3\lambda}.
\end{align}
The bound in \eqref{10.1} follows.

(2) If $n_2\neq 0$ and $\nu_p(n_2) < \lambda/2$, we let $\gamma=\overline{\widehat{c_1}+n_2d_2}$. Then
the three equations give
\begin{align}\label{6.8}
\left\{\begin{array}{l}
\beta_1\equiv \widehat{c_2}^2\gamma^2\beta_2\bmod p^{\lambda/2},\\
\gamma\equiv\overline{\widehat{c_1}}\big(1+\theta_1\widehat{c_1}
\overline{\widehat{c_2}} n_2
                \beta_1^2 \big)\bmod p^{\lambda/2},\\
\overline{\gamma}\equiv
\widehat{c_1}\big(1-\eta_1\overline{\widehat{c_1}}
\widehat{c_2}n_2\beta_2^2\big)\bmod p^{\lambda/2}.
\end{array}\right.
\end{align}
Plugging the second equation into the first equation in \eqref{6.8} we get
\begin{align*}
\beta_1\equiv \widehat{c_2}^2\overline{\widehat{c_1}}^2
\big(1+\theta_1\widehat{c_1}
\overline{\widehat{c_2}} n_2
                \beta_1^2\big)^2\beta_2\bmod p^{\lambda/2}.
\end{align*}
By the above equation and the last two equations in \eqref{6.8} we get
\begin{equation}\label{6.10}
\begin{split}
 \left(\theta_1\widehat{c_1}\overline{\widehat{c_2}}\right)^5\xi^5&
   +4\left(\theta_1\widehat{c_1}\overline{\widehat{c_2}}\right)^4\xi^4
   +6\left(\theta_1\widehat{c_1}\overline{\widehat{c_2}}\right)^3\xi^3
   +4\left(\theta_1\widehat{c_1}\overline{\widehat{c_2}}\right)^2\xi^2\\
&-\theta_1\eta_1\widehat{c_1}^4\overline{\widehat{c_2}}^4\xi^2
   +\theta_1\widehat{c_1}\overline{\widehat{c_2}}\xi
   -\eta_1\widehat{c_1}^3\overline{\widehat{c_2}}^3\xi
   \equiv 0 \bmod p^{\lambda/2},
\end{split}
\end{equation}
where $\xi=n_2\beta_1^2$. Thus there are at
most $5$ roots modulo $p^{\lambda/2}$ for $\xi$.
Invoking Hensel's lemma, we see that
there are at most $10$ solutions modulo $p^{\lambda/2-\nu_p(n_2)}$ for $\beta_1$.
For fixed $\xi$,
$\gamma$ is uniquely determined modulo $p^{\lambda/2}$ and for fixed $\gamma$ and $\beta_1$,
$\beta_2$ is uniquely determined modulo $p^{\lambda/2}$ by the first equation in \eqref{6.8}.
Then by the last congruence equation in \eqref{ff}, $d_2$ is uniquely determined modulo $p^{\lambda/2}$.
Therefore,
\begin{align*}
\mathfrak{C}_2(n_2)\ll p^{5\lambda/2+\nu_p(n_2)}.
\end{align*}
By \eqref{6.7} and \eqref{6.10}, the bound in \eqref{10.2} follows.
When $\lambda$ is odd, we use $p^{(\lambda+1)/2}$ instead of $p^{\lambda/2}$ in \eqref{aidecomp0} and proceed identically. This will clearly result in an extra factor of $p^{3/2}$ in the final estimate as indicated in the statement of the lemma.
\end{proof}

\subsection{Completion of the proof}

We treat the cases where $n_2=0$ and $n_2\neq 0$ separately
and denote their contributions to $\mathbf{\Omega}$ in \eqref{Omega sum} by
$\mathbf{\Omega}_0$ and $\mathbf{\Omega}_{\neq 0}$, respectively.

For $n_2=0$, we necessarily have $c_1=c_2=c$.
For some $\alpha\geq 1$, $p^\beta\|m_1$ and $p^\beta\|m_2$ with $\beta<\lceil\frac{\alpha}{4}\rceil$,
the solutions of $m_1^4\equiv m_2^4 \bmod  p^{\alpha}$ are $m_1=\alpha m_2$ where
$\alpha^4\equiv 1 \bmod  p^{\alpha-4\beta}$. There are at most $4$ different solutions of $x^4-1=0$
in $\mathbb{Z}_p$ for $p\geq 2$. Then using Hensel's lemma, there are at most $4$ different solutions of $x^4-1=0$
in $\mathbb{Z}_{p^{\alpha-4\beta}}$. The solutions of $m_1^4\equiv m_2^4 \equiv 0 \bmod  p^{\alpha}$ are $m_1=v_1p^{\lceil\frac{\alpha}{4}\rceil}$, $m_2=v_2p^{\lceil\frac{\alpha}{4}\rceil}$ for some $v_1, v_2 \in \mathbb{Z}_{p^{\alpha-4\beta}}$.

Splitting the sum over $m_1$ and $m_2$ according as $m_1=m_2$ or not,
and applying Lemma \ref{integral:lemma} and Lemma \ref{character C2} (1), we have
(set $v=\lceil \lambda/2 \rceil$)
\begin{equation*}
  \begin{split}
  \mathbf{\Omega}_0 \ll& \frac{1}{C}\sum_{\substack{1\leq c \leq C \\ (c,p)=1}}\frac{1}{c}
\sum_{0\leq \beta<\lceil\frac{\lambda-v}{4}\rceil}
\sum_{\substack{1\leq |m_1| \leq N^\varepsilon p^k/Cp^\lambda \\ (m_1,c)=1,\ p^\beta\|m_1}}
\sum_{\substack{1\leq |m_2| \leq N^\varepsilon p^k/Cp^\lambda\\ m_1^4\equiv m_2^4 \bmod p^{\lambda/2}}}
\frac{1}{\widehat{c}^2p^{\lambda}}
\widehat{c}^2(\widehat{c},m_1-m_2)p^{3\lambda} \\
&+\frac{1}{C}\sum_{\substack{1\leq c \leq C \\ (c,p)=1}}\frac{1}{c}
\sum_{\substack{1\leq |m_1| \leq N^\varepsilon p^k/Cp^\lambda \\ (m_1,c)=1,\ p^{\lceil\frac{\lambda-v}{4}\rceil}\mid m_1}}
\sum_{\substack{1\leq |m_2| \leq N^\varepsilon p^k/Cp^\lambda
\\ (m_1,c)=1,\ p^{\lceil\frac{\lambda-v}{4}\rceil}\mid m_1}}
\frac{1}{\widehat{c}^2p^{\lambda}}
\widehat{c}^2(\widehat{c},m_1-m_2)p^{3\lambda} \\
\ll& \frac{N^\varepsilon p^{k+\lambda}}{C^2}\Big(1+\frac{p^{k-2\lambda+v}}{C}\Big)
+\frac{N^\varepsilon p^{2k-2\lceil\frac{\lambda-v}{4}\rceil}}{C^3}\\
\ll& \frac{p^{k+\lambda}}{C^2}+\frac{p^{2k-\lambda/4}}{C^3}.
   \end{split}
\end{equation*}
Then we deal with the case $n_2 \neq 0$.
By Lemma \ref{integral:lemma} and Lemma \ref{character C2} (2), we get
\begin{equation*}
  \begin{split}
     \mathbf{\Omega}_{\neq 0}\ll&\; \frac{1}{C}\sum_{\substack{1\leq c_1 \leq C \\ (c_1,p)=1}}\frac{1}{c_1}
\sum_{\substack{1\leq c_2 \leq C \\ (c_2,p)=1}}\frac{1}{c_2}
\sum_{\substack{1\leq |m_1| \leq N^\varepsilon p^k/Cp^\lambda \\ (m_1,c_1)=1}}
\sum_{\substack{1\leq |m_2| \leq N^\varepsilon p^k/Cp^\lambda \\ (m_2,c_2)=1}} \frac{1}{\widehat{c_1}\widehat{c_2}p^{\lambda}}\\
&\;\times\sum_{n_2\ll N^\varepsilon C^2p^\lambda/N_1}
\widehat{c_1}\widehat{c_2}(\widehat{c_1},\widehat{c_2},n_2)
p^{\lceil5\lambda/2\rceil+\min\{\nu_p(n_2), \lceil\lambda/2\rceil\}}\\
\ll&\;p^{3/2}N^{\varepsilon}\frac{p^{2k+3\lambda/2}}{CN_1}.
   \end{split}
\end{equation*}
Recall that $1\ll N_1\ll N^{2+\varepsilon}/C^3$ and $C=\sqrt{N/p^{\lambda}}$.
We further imply
\begin{align*}
\widetilde{S}_1(N,N_1)&\;\ll\frac{N^{3/2}N_1^{1/2}}{Cp^{(k+3\lambda)/2}}
  \left(\frac{p^{k+\lambda}}{C^2}
  +\frac{p^{2k-\lambda/4}}{C^3}
  +p^{3/2}N^{\varepsilon}\frac{p^{2k+3\lambda/2}}{CN_1}\right)^{1/2}\\
  &\;\ll N^{3/4+\varepsilon}p^{3\lambda/4}+N^{1/2+\varepsilon}p^{k/2+3\lambda/8}
  +p^{3/4}N^{3/4+\varepsilon}p^{k/2-\lambda/2}.
\end{align*}
Taking $\lambda=\lfloor2k/5\rfloor$, we get
\begin{equation*}
\widetilde{S}_1(N,N_1)\ll p^{3/4}N^{3/4+\varepsilon}q^{3/10}+N^{1/2+\varepsilon}q^{13/20}.
\end{equation*}
As we point out in Remark~\ref{remark5.1} and Remark~\ref{remark5.2}, all the other cases
are similar and in fact easier. Hence, we finally
prove Theorem \ref{Bound theorem}.


\end{document}